\documentclass[11pt]{article}
\usepackage[margin=1in]{geometry}
\usepackage{graphicx,xcolor,cite,mathtools}
\usepackage{amssymb,amsmath,amsthm,mathrsfs,paralist,bm,esint,setspace}
\RequirePackage[colorlinks,citecolor=blue,urlcolor=blue]{hyperref}

\newcommand{\R}{{\mathbb{R}}}
\newcommand{\E}{\mathrm{E}}

\renewcommand{\P}{\mathrm{P}}
\renewcommand{\d}{\mathrm{d}}
\newcommand{\<}{\langle}
\renewcommand{\>}{\rangle}
\newcommand{\e}{\mathrm{e}}

\newcommand{\Var}{\text{\rm Var}}
\newcommand{\lip}{\text{\rm Lip}}

\DeclareMathOperator{\Cov}{\text{\rm Cov}}

\input cyracc.def

\title{Central limit theorems for parabolic stochastic partial
	differential equations\thanks{%
	Research supported in part by  NSF grants DMS-1811181 (D.N.) and DMS-1855439 (D.K.).}}
\author{Le Chen\\Emory University\\\texttt{le.chen@emory.edu}\\
	\and
		Davar Khoshnevisan\\University of Utah\\\texttt{davar@math.utah.edu}\\
	\and\and
		David Nualart\\University of Kansas\\\texttt{nualart@ku.edu}\\
	\and
		Fei Pu\\University of Utah\\\texttt{pu@math.utah.edu}
	}
\date{May 10, 2021}

\begin{document}
\newtheorem{stat}{Statement}[section]
\newtheorem{proposition}[stat]{Proposition}
\newtheorem*{prop}{Proposition}
\newtheorem{corollary}[stat]{Corollary}
\newtheorem{theorem}[stat]{Theorem}
\newtheorem{lemma}[stat]{Lemma}
\theoremstyle{definition}
\newtheorem{definition}[stat]{Definition}
\newtheorem*{cremark}{Remark}
\newtheorem{remark}[stat]{Remark}
\newtheorem*{OP}{Open Problem}
\newtheorem{example}[stat]{Example}
\newtheorem{nota}[stat]{Notation}

\numberwithin{equation}{section}
\maketitle

\begin{abstract}
Let $\{u(t\,,x)\}_{t\ge 0, x\in \R^d}$ denote  the solution of a $d$-dimensional
nonlinear stochastic heat equation that is driven by a Gaussian noise, white in time with
a homogeneous spatial covariance that is
a finite Borel measure $f$ and satisfies Dalang's condition.
We prove two general functional central limit theorems for occupation fields  of  the form
$N^{-d} \int_{\R^d} g(u(t\,,x)) \psi(x/N)\,\d x$ as $N\rightarrow \infty$,
where $g$ runs over the class of Lipschitz functions on $\R^d$
and $\psi\in L^2(\R^d)$. The proof uses Poincar\'e-type inequalities,
Malliavin calculus, compactness arguments, and  Paul L\'evy's classical characterization of
Brownian motion as the only mean zero, continuous L\'evy process.
Our result generalizes  central limit theorems of Huang et al \cite{HuangNualartViitasaari2018,HuangNualartViitasaariZheng2019} valid when
$g(u)=u$ and $\psi = \mathbf{1}_{[0,1]^d}$.

\begin{center}
 \textbf{R\'esum\'e}
 \end{center}

Soit $\{u(t\,,x)\}_{t\ge 0, x\in \R^d}$ la solution d’une \'equation de la chaleur stochastique non-lin\'eaire $d$-
dimensionnelle, perturb\'ee par un bruit  gaussien,
blanc en temps et avec une covariance homog\`ene en espace  donn\'ee par une mesure de Borel finie qui satisfait la condition de Dalang.
Nous  d\'emontrons deux th\'eor\`emes de la  limite centrale fonctionnels pour des champs d’occupation de la forme
$N^{-d} \int_{\R^d} g(u(t\,,x)) \psi(x/N)\,\d x$ quand $N\rightarrow \infty$,
o\`u $g$ est une function lipschitzienne sur $\R^d$ et $\psi\in L^2(\R^d)$. La preuve utilise des inegalit\'es de Poincar\'e,
le calcul de Malliavin, des arguments de compacit\'e et la caract\'erisation du mouvement brownien comme le seul processus de L\'evy continu avec moyenne z\'ero.  Notre r\'esultat g\'en\'eralise les th\'eor\`emes de la  limite centrale  de Huang et al \cite{HuangNualartViitasaari2018,HuangNualartViitasaariZheng2019} qui sont valables pour  $g(u)=u$ et $\psi = \mathbf{1}_{[0,1]^d}$.

\end{abstract}

\bigskip\bigskip

\bigskip

\bigskip

\noindent{\it Running head:} CLT for SPDEs.

\newpage
{
 \hypersetup{linkcolor=black}
  \tableofcontents
}

\section{Introduction}

Consider the stochastic PDE
\begin{equation}\label{SHE}
	\partial_t u = \tfrac12\Delta u + \sigma(u)\eta\qquad\text{on $(0\,,\infty)\times\R^d$},
\end{equation}
subject to $u(0)\equiv1$ on $\R^d$, where $\sigma:\R\to\R$ is non random and Lipschitz continuous,
and $\eta$ denotes a centered, generalized Gaussian field whose covariance form is described formally as
\[
	\Cov[\eta(t\,,x)\,,\eta(s\,,y)] = \delta_0(t-s)f(x-y)
	\qquad\text{for all $s,t\ge0$ and $x,y\in\R^d$},
\]
for a nonnegative-definite tempered Borel measure $f$ on $\R^d$ that we fix
throughout.	To avoid triviality, throughout this paper, we assume that
	\[
		\sigma(1)\ne 0.
		\]

Somewhat more formally,
this means that the Wiener-integrals
\begin{equation}\label{W}
	W_t(\phi):=\int_{(0,t)\times\R^d}
	\phi(x)\,\eta(\d r\,\d x)
	\qquad[t\ge0, \phi\in\mathscr{S}(\R^d)]
\end{equation}
define a centered Gaussian random field with covariance,
\[
	\Cov\left[W_s(\phi_1)\,,W_t(\phi_2)\right] = (s\wedge t)
	\< \phi_1\,,\phi_2*f \>_{L^2(\R^d)}
	\qquad\text{for all $s,t\ge0$ and $\phi_1,\phi_2\in\mathscr{S}(\R^d)$},
\]
where $\mathscr{S}(\R^d)$ denotes the space of Schwartz test functions.
Thus, we can (and will) think of $\{W_t\}_{t\ge0}$ as an infinite-dimensional
Brownian motion.

Dalang \cite{Dalang1999} has proved that \eqref{SHE} has a mild solution $u$ provided that\footnote{%
	Our Fourier transform is normalized
	so that $\hat{h}(z) = \int_{\R^d}\e^{ix\cdot z}h(x)\,\d x$ for all $h\in L^1(\R^d)$
	and $z\in\R^d$.}
\begin{equation}\label{Dalang}
	\Upsilon(\lambda) := \frac{2}{(2\pi)^d}\int_{\R^d} \frac{\hat{f}(\d z)}{2\lambda+\|z\|^2}<\infty,
\end{equation}
for one, hence all, $\lambda>0$;\footnote{Caveat: Our $\Upsilon(\lambda)$ is
equal to Foondun and Khoshnevisan's $2\Upsilon(\lambda/2)$
\cite{FoondunKhoshnevisan2013}. The slight alteration
of this notation should not cause any confusion.} moreover, Dalang ({\it loc.\ cit.})
has proved that $\R_+\times\R^d\ni (t\,,x)\mapsto u(t\,,x)$ is the only predictable
random field that is continuous in  $L^k(\Omega)$ for
every  $k\ge2$. Condition \eqref{Dalang} will be in force from now on in
order to guarantee that \eqref{SHE} is well posed.

In a companion paper \cite{CKNP} we  examine the ergodic-theoretic properties
of the spatial random field $u(t)=\{u(t\,,x)\}_{x\in\R^d}$
for all $t>0$. Specifically, we  prove in \cite{CKNP} that:
\begin{compactenum}
	\item For every $t> 0$, $u(t)$ is stationary and it is  ergodic
		  if $\hat{f}\{0\}=0$. Moreover,
		the following conditions are equivalent:
		\begin{compactenum}
			\item $\hat{f}\{0\}=0$;
			\item $\hat{f}$ has no atoms;
			\item $f\{x\in\R^d:\, \|x\|<r\}=o(r^d)$ as $r\to\infty$;
			\item $u(t)$ is ergodic for all $t\ge0$ in the case that $\sigma$ is a non-zero constant;
		\end{compactenum}
	\item $u(t)$ is (weakly) mixing for every $t>0$ if
		\begin{equation}\label{eq:mixing}
			\lim_{\|x\|\to\infty} \int_{\R^d} \frac{\e^{ix\cdot z}\,\hat{f}(\d z)}{
			2\lambda+\|z\|^2}=0.
		\end{equation}
	\item Condition \eqref{eq:mixing} is necessary and sufficient for $u(t)$ to be mixing
		(for every $t>0$) in the case that $\sigma$ is a constant.
\end{compactenum}
When $\sigma$ is a non-zero constant, parts of these results simplify to well-known
ergodic-theoretic facts about stationary Gaussian processes. Specifically,
the equivalence of items 1(b) and 1(d), as well as the validity of item
3, can be found in the classical work of
Maruyama \cite{Maruyama}; see also the subsequent
exposition of Dym and McKean \cite{DymMcKean}.

As was  mentioned, $u(t)$ is ergodic for all $t>0$ if
\begin{equation}\label{fhat:0}
	\hat{f}\{0\}=0,
\end{equation}
and hence by the ergodic theorem,
\begin{equation}\label{ergodic:thm}
	\lim_{N\to\infty} \frac{1}{N^d}\int_{[0,N]^d} g(u(t\,,x))\,\d x= \E[g(u(t\,,0))]
	\qquad\text{a.s.\ for all $g\in\lip$ and $t>0$},
\end{equation}
where $\lip$ denotes the collection of all real-valued
Lipschitz-continuous functions on $\R$. The purpose of the present
article is to determine whether, and when,
\eqref{ergodic:thm} has a matching central limit theorem (CLT). In special cases ---
particularly when $g$ is linear --- such CLTs have recently been studied in Huang et al
\cite{HuangNualartViitasaari2018,HuangNualartViitasaariZheng2019}.
Our main goal is to study the non-linear case. Although our methods (see the description next paragraph) differ from
those of Huang et al ({\it ibid.}) where the Malliavin-Stein method is appealed to estimate the total variation distance, a common point is crucial use of the Malliavin calculus.

Because weak mixing implies ergodicity, it follows immediately from the above remarks
(in the case that $\sigma$ is constant) that \eqref{eq:mixing}
is a little stronger than \eqref{fhat:0}. It is also well known that mixing is by
itself not enough to ensure a CLT. Strong mixing, however,
\emph{can} imply a CLT
(see Bradley \cite{Bradley}). Unfortunately, we are not able to determine precise conditions that
ensure the strong mixing of $u(t)$. Thus, we are forced to introduce novel methods in order
to establish the existence of a CLT: By contrast with ``mixing and blocking arguments''
of the literature on strong mixing, we use Malliavin's calculus,
Poincar\'e-type inequalities, compactness arguments, and Paul L\'evy's
characterization theorem of Brownian motion as the only mean-zero, continuous L\'evy process.

Throughout, we assume that
\begin{equation}\label{f:finite}
	0< f(\R^d)<\infty.
\end{equation}
The positivity of the total mass of $f$ merely ensures non triviality.
After all, if $f(\R^d)=0$ then \eqref{SHE} is deterministic and there is nothing
left to study.
The more interesting finite-mass condition on $f$ turns out to be unimprovable
and is a slightly stronger condition than the mixing condition \eqref{eq:mixing}.
In order to see why, note that because of \eqref{f:finite} the Fourier transform of $f$ is
a uniformly bounded and continuous function defined by
\[
	\hat{f}(z) = \int_{\R^d} \e^{ix\cdot z}\,f(\d x)\qquad\text{for all $z\in\R^d$}.
\]
Therefore, \eqref{eq:mixing} is a consequence of the Riemann--Lebesgue lemma
and Dalang's condition \eqref{Dalang}.

The following summarizes our main finding in its simplest form.
\begin{theorem}\label{th:main}
	Choose and fix $t>0$ and $g\in\lip$, and
	suppose \eqref{f:finite} holds. Then,
	\begin{equation*}
		N^{d/2}\left(\frac{1}{N^d}\int_{[0,N]^d} g(u(t\,,x))\,\d x
		-\E[g(u(t\,,0))] \right)
		\xrightarrow{\text{\rm d}\,} X
		\qquad\text{as $N\to\infty$},\tag{CLT}\label{CLT}
	\end{equation*}
	where $X=X(t\,,g)$ has a centered normal distribution,
	and the symbol $\xrightarrow{\text{\rm d}\,}$ refers to
	convergence in distribution.
	Moreover, \eqref{CLT} is equivalent to the condition $f(\R^d)<\infty$ when $\sigma$
	is a constant.
\end{theorem}
Although it is not so easy to prove Theorem \ref{th:main} directly, it turns out to be
possible to give a relatively simple proof of a
much more general result (Theorem \ref{th:1.1}).
In order to describe the more general result we need to abstract the problem
to a suitable level, and that requires some work which we relegate to the next section.
Moreover, we remark that the central limit theorem may still hold with $f(\R^d)=\infty$  with different scaling, for instance, $f(\d x)=\|x\|^{-\beta}\d x$ as considered in  \cite{HuangNualartViitasaariZheng2019}.

Let us conclude the Introduction by setting forth some notation that will be used
throughout. Throughout, let
$\mathcal{F}=\{\mathcal{F}_t\}_{t\ge0}$ denote the Brownian filtration
generated by the infinite-dimensional Brownian motion $\{W_t\}_{t\ge0}$ of \eqref{W},
and  assume that $\mathcal{F}$ is augmented in the usual way.
We write ``$g_1(x)\lesssim g_2(x)$ for all $x\in X$'' when there exists a real number
$L$ such that $g_1(x)\le Lg_2(x)$ for all $x\in X$.
Alternatively, we might write
``$g_2(x)\gtrsim g_1(x)$
for all $x\in X$.'' By ``$g_1(x)\asymp g_2(x)$ for all $x\in X$'' we mean that
$g_1(x)\lesssim g_2(x)$ for all $x\in X$ and $g_2(x)\lesssim g_1(x)$ for all $x\in X$.
Finally,
``$g_1(x)\propto g_2(x)$ for all $x\in X$'' means that there exists a real number $L$
such that $g_1(x)=L g_2(x)$ for all $x\in X$.
For every $Z\in L^k(\Omega)$, we write $\|Z\|_k$ instead of the more cumbersome
$\|Z\|_{L^k(\Omega)}$. Set
\[
	\lip(g) := \sup_{a,b\in\R}
	\frac{|g(b)-g(a)|}{|b-a|},
\]
where $0\div 0 :=0$. Thus, $g\in\lip$ if and only if $\lip(g)<\infty$.

\section{Main results}

Before we describe the main results of this paper we introduce
the occupation fields of the processes  $u(t)$, for every $t>0$, where we recall
$u$ denotes the solution to the SPDE \eqref{SHE}.

\subsection{The occupation field}
Choose and fix some $t\ge0$, and consider the collection of all random variables of the
form
\begin{equation}\label{S}
	\mathcal{S}_{N,t}(\psi\,,g) := \int_{\R^d} g(u(t\,,x))\psi_N(x)\,\d x -
	\E [ g(u(t\,,0)) ]\int_{\R^d}\psi(x)\,\d x,
\end{equation}
as $N>0$ ranges over all positive reals, $g$ ranges over all Lipschitz functions, and
\begin{equation}\label{psi_N}
	\psi_N(x) := N^{-d}\psi(x/N)\qquad\text{for all $x\in\R^d$ and $N>0$},
\end{equation}
for a sufficiently-large family of ``nice'' functions $\psi:\R^d\to\R$.\footnote{%
	We will introduce many other functions with many other
	subscripts. The subscript ``$N$'' is however reserved for the notation in \eqref{psi_N}.
}
The left-hand side of \eqref{CLT} is equal to
$N^{d/2} \mathcal{S}_{N,t}(\bm{1}_{[0,1]^d},g)$,
but it turns out to be easier to study the CLT for $\mathcal{S}_{N,t}(\psi\,,g)$
for more general functions $\psi$ than just $\psi=\bm{1}_{[0,1]^d}$.

As was mentioned in the Introduction, Dalang \cite{Dalang1999}
has proved that condition \eqref{Dalang} (which is enforced throughout this paper)
implies among other things that $u$
is continuous in $L^k(\Omega)$ for every $k\ge2$.
This means that
\[
	\lim_{(s,y)\to(t,x)}
	\left\| u(s\,,y) - u(t\,,x)\right\|_k=0\qquad\text{for all $k\ge2$, $t\ge0$, and $x\in\R^d$}.
\]
A small extension of Doob's separability theory \cite{Doob} then implies
that there exists a version of $u$, which we continue to denote by $u$,
such that $\R_+\times\R^d\times\Omega\ni
(t\,,x\,,\omega)\mapsto u(t\,,x)(\omega)$
is measurable. Therefore, \eqref{S} and Fubini's theorem yield a well-defined
stochastic process provided that
\[
	\int_{\R^d}\E\left(|g(u(t\,,x))|\right) |\psi_N(x)|\,\d x<\infty\quad\text{
	for all $t,N>0$ and $g\in\lip$.}
\]
Since $u(t)$ is stationary, the preceding integral simplifies to
\[
	\E\left(|g(u(t\,,0))|\right) \|\psi\|_{L^1(\R^d)}
	\le\left(  |g(0)|+\lip(g)\E(|u(t\,,0)|)\right) \|\psi\|_{L^1(\R^d)},
\]
which is finite locally uniformly in $t\ge0$ provided that $\psi\in L^1(\R^d)$.
In this way we see that the random field
\[
	\left\{ \mathcal{S}_{N,t}(\psi\,,g)\,;N>0,\,\psi\in L^1(\R^d)\,,g\in\lip\right\}
\]
is well defined for every $t\ge0$.

The following is one of the main technical innovations of this paper.
Before we state this result, note that because $f(\R^d)>0$ the function
$\Upsilon$ defined in \eqref{Dalang} is strictly decreasing on $(0\,,\infty)$.
Therefore, it has an inverse which we denote by
\begin{equation}\label{Lambda:Upsilon}
	\Lambda := \Upsilon^{-1}.
\end{equation}

\begin{theorem}\label{th:S}
	For all real numbers $N,T>0$, $\varepsilon\in(0\,,1)$, and $k\ge2$,
	and for every pair of non-random functions
	$\psi\in L^1(\R^d)\cap L^2(\R^d)$ and $g\in\lip$,
	\begin{equation}\label{eq:S}
		\sup_{t\in(0,T)}\left\| \mathcal{S}_{N,t}(\psi\,,g) \right\|_k
		\le \frac{\mathsf{A}(\varepsilon)\sqrt{T k}}{N^{d/2}}
		\exp\left\{ 2T\Lambda\left(
		\frac{\mathsf{a}(\varepsilon)}{k}\right)\right\}
		\, \lip(g)\|\psi\|_{L^2(\R^d)},
	\end{equation}
	where
	\begin{equation}\label{eq:Aa}
		\mathsf{A}(\varepsilon) :=
		\frac{{ 16}[|\sigma(0)|\vee\lip(\sigma)]\sqrt{f(\R^d)}}{ {\varepsilon^{3/2}}},
		\qquad
		\mathsf{a}(\varepsilon)
		:= \frac{(1-\varepsilon)^2}{2^{(d+6)/2}[|\sigma(0)|\vee\lip(\sigma)]^2}.
	\end{equation}
\end{theorem}

The proof of Theorem \ref{th:S} hinges on careful analysis of a Poincar\'e inequality for
the infinite-dimensional Brownian motion $W$ defined in \eqref{W}, and the statement of
Theorem \ref{th:S} has a number of consequences for the present work.
We mention one of them next.

For every $\psi\in L^2(\R^d)$ we can find $\psi^1,\psi^2,\ldots\in L^1(\R^d)\cap L^2(\R^d)$
such that $\psi^n\to\psi$ in $L^2(\R^d)$ as $n\to\infty$. Because
every $\mathcal{S}_{N,t}$ is a random linear functional on $L^1(\R^d)\times\lip$, it follows readily
from \eqref{eq:S} that $\{\mathcal{S}_{N,t}(\psi^n\!,g)\}_{n=1}^\infty$ is a Cauchy sequence
in $L^k(\Omega)$ for every $k\ge2$. Consequently,
\[
	\mathcal{S}_{N,t}(\psi\,,g) := \lim_{n\to\infty} \mathcal{S}_{N,t}(\psi^n\!,g)
	\quad\text{exists in $L^k(\Omega)$ for every $k\ge2$}.
\]
The construction of $\mathcal{S}_{N,t}(\psi\,,g)$
does not depend on the particular sequence $\{\psi^n\}_{n=1}^\infty$,
and $\mathcal{S}_{N,t}(\psi\,,g)$ continues to satisfy \eqref{eq:S}. Moreover, every
$\mathcal{S}_{N,t}$ is a random linear functional on $L^2(\R^d)\times\lip$.

\begin{definition}\label{def:S}
	Fix some $t\ge0$.
	By the \emph{occupation}, or \emph{sojourn}, \emph{field} of $u(t)$ we mean the
	above-defined
	random field $\mathcal{S}[t]:=
	\{\mathcal{S}_{N,t}(\psi\,,g);\, N>0,\, \psi\in L^2(\R^d),\, g\in\lip\}$.
\end{definition}

Definition \ref{def:S} has non-trivial content
since $\mathcal{S}_{N,t}(\psi\,,g)$ cannot be defined pathwise
when $\psi\in L^2(\R^d)$. Nor can we claim that $\mathcal{S}_{N,t}(\psi\,,g)$
satisfies \eqref{S}
when $\psi\in L^2(\R^d)$. The situation is somewhat akin to what happens in the
construction of the Fourier transform on
$\R^d$. In that setting, $\hat{\phi}(z)$ is simply equal to the Lebesgue integral
$\int_{\R^d}\exp(ix\cdot z)\phi(x)\,\d x$
when $\phi\in L^1(\R^d)$, but not when $\phi\in L^2(\R^d)\setminus L^1(\R^d)$.
That is, unless we interpret the integral $\int_{\R^d}\exp(ix\cdot z)\phi(x)\,\d x$
suitably in order to remove all singularities that arise when $\phi\in L^2(\R^d)\setminus
L^1(\R^d)$. Thus, we can see that Theorem \ref{th:S} is ``removing the singularities''
that arise when we transition from $\psi\in L^1(\R^d)$ to $\psi\in L^2(\R^d)$.

\subsection{Functional CLTs}

Now that the occupation fields $\{\mathcal{S}[t]\}_{t\ge0}$
has been properly constructed we can describe
the main two results of this paper. These are two functional CLTs, the first of
which is the following.

\begin{theorem}\label{th:1.1}
	Choose and fix $t\ge0$ and $g\in\lip$. Also, let $\mathscr{F}\subset L^2(\R^d)$
	be a compact set such that
	$\int_0^1 [ \bm{N}_{\mathscr{F}\!,L^2(\R^d)}(r)]^\varepsilon\,\d r<\infty$
	for some $\varepsilon>0$,
	where $\bm{N}_{\mathscr{F}\!,L^2(\R^d)}$ denotes the metric entropy of
	$\mathscr{F}$ in the metric defined by the norm of $L^2(\R^d)$ [\S\ref{subsec:ME}].
	Then, we have the functional CLT,
	\[
		\left\{ N^{d/2} \mathcal{S}_{N,t}(\psi\,,g)\,; \psi\in\mathscr{F}\right\}
		\xrightarrow{C(L^2(\R^d))}
		\left\{ \Gamma_t(\psi\,,g);\, \psi\in\mathscr{F}\right\}\qquad
		\text{as $N\to\infty$},
	\]
	where $\Gamma_t=\{\Gamma_t(\psi\,,g);\,\psi\in L^2(\R^d),g\in\lip\}$
	is a centered Gaussian random field whose covariance function is
	\begin{equation}\label{Cov:Gamma}
		\Cov\left[ \Gamma_t(\psi\,,g) \,, \Gamma_t(\Psi\,,G)\right]=
		\mathbf{B}_t(g\,,G)\cdot \<\psi\,,\Psi\>_{L^2(\R^d)},
	\end{equation}
	for every $\psi,\Psi\in L^2(\R^d)$ and $g,G\in\lip$.
	The bilinear form $\mathbf{B}_t:\lip\times\lip\to\R$ is non-negative definite
	and defined in \eqref{def:B_t} below.
\end{theorem}
Examples \ref{ex:BS:1} and \ref{ex:BS:2} can be combined to produced a
number of compact sets $\mathscr{F}\subset L^2(\R^d)$ to which Theorem
\ref{th:1.1} applies. For now, let us mention the following (see Example \ref{ex:BS:1}),
which immediately  implies \eqref{CLT}, the main part of Proposition \ref{pr:1}. The remainder
of Proposition \ref{pr:1} is not hard to prove; the details can be found
in \S\ref{subsec:T1:2} below.

Define
\[
	[0\,,z] := [0\,,z_1]\times\cdots\times[0\,,z_d]
	\qquad\text{for all $z\in\R^d_+$}.
\]
Then, for every fixed $t,m\ge0$ and $g\in\lip$, the random field
$W_{N,t}:=\{W_{N,t}(y);\, y\in[0\,,m]^d\}$, defined by
\begin{align}\label{W_{N,t}}
	W_{N,t}(y) := N^{d/2}\left(
	\frac{1}{N^d}\int_{[0,Ny]} g(u(t\,,x))\,\d x - \E[g(u(t\,,0))]\prod_{j=1}^d y_j
	\right),
\end{align}
converges weakly in $C([0\,,m]^d)$ to $\{\sqrt{\mathbf{B}_t(g\,,g)}W(y)\,; y\in[0\,,m]^d\}$
as $N\to\infty$, where $W$ denotes a $d$-parameter, standard Brownian sheet
indexed by $[0\,,m]^d$ (see Walsh \cite{Walsh}).

The proof of Theorem \ref{th:1.1} produces at
no extra cost a second functional CLT that we describe next.
We can view the space $\lip$ as a separable metric space, once it is endowed with
the metric defined by the norm,
\begin{align}\label{E:NormLip}
	\|g\|_{\lip} := |g(0)| + \lip(g)
	\qquad\text{for all $g\in\lip$}.
\end{align}

With this in mind, we have the following.

\begin{theorem}\label{th:1.2}
	Choose and fix $t\ge0$ and $\psi\in L^2(\R^d)$. Also, let $\mathscr{G}\subset \lip$
	be a separable and compact set such that
	$\int_0^1 [ \bm{N}_{\mathscr{G}\!,\lip}(r)]^\varepsilon\,\d r<\infty$
	for some $\varepsilon>0$,
	where $\bm{N}_{\mathscr{G}\!,\lip}$ denotes the metric entropy of
	$\mathscr{G}$ in the metric defined by the norm of $\lip$.
	Then, we have the functional CLT,
	\[
		\left\{ N^{d/2} \mathcal{S}_{N,t}(\psi\,,g)\,; g\in\mathscr{G}\right\}
		\xrightarrow{C(\lip)}
		\left\{ \Gamma_t(\psi\,,g);\, g\in\mathscr{G}\right\}\qquad
		\text{as $N\to\infty$},
	\]
	for the same Gaussian random field $\Gamma_t$ that appeared in Theorem \ref{th:1.1}.
\end{theorem}
Examples \ref{ex:Lip:1} and \ref{ex:Lip:2} can be combined to
create examples of compact sets $\mathscr{G}$ to which Theorem \ref{th:1.2}
applies.

Finally let us conclude this section with a closing remark.

\begin{remark}
         \begin{itemize}
         \item [(1)]
	It is easy to see from \eqref{Cov:Gamma}
	that $\Gamma_t$ is a random
	bilinear mapping for every $t\ge0$; that is, for all $\alpha_1,\ldots,\alpha_m,
	\beta_1,\ldots,\beta_n\in\R$, $\psi^1,\ldots,\psi^m\in
	L^2(\R^d)$, and $g^1,\ldots,g^n\in\lip$,
	\[
		\Gamma_t\left( \alpha_1\psi^1+\cdots+\alpha_m\psi^m,
		\beta_1g^1+\cdots+\beta_n g^n\right)
		=\sum_{i=1}^m\sum_{j=1}^n \alpha_i\beta_j\Gamma_t\left( \psi^i\,,g^j\right)
		\qquad\text{a.s.}
	\]
	To prove this, we simply compute the variance of the difference of the two sides,
	and note that the said variance is zero. The details are elementary,
	and therefore omitted.
	\item [(2)] We point out that as a process in time, a functional CLT is proved in \cite[Theorem 2.3]{CKNP_d} using Malliavin-Stein method provided that $f$ satisfies  \eqref{f:finite} and the reinforced Dalang's condition (see \cite[(1.6)]{CKNP_d}).
	It is also possible to consider the convergence of $N^{d/2} \mathcal{S}_{N,t}(\bf{1}_{[0, x]\times [0, 1]^{d-1}}\,,g)$ as a function of $(t, x)$. We leave it for interested reader.
	
	\end{itemize}
\end{remark}

\section{Preliminaries}

We begin the work by briefly collecting and developing some notation and basic background
information that will be used tacitly throughout the remainder of this paper.

\subsection{Potential theory}
Define, for every $t,\lambda>0$ and $x\in\R^d$,
\begin{equation}\label{p:v}
	\bm{p}_t(x) = \frac{1}{(2\pi t)^{d/2}}\exp\left( - \frac{\|x\|^2}{2t}\right)
	\quad\text{and}\quad
	\bm{v}_\lambda(x) = \int_0^\infty\e^{-\lambda s}\bm{p}_s(x)\,\d s.
\end{equation}
The notation should not be misunderstood with our convention in \eqref{psi_N},
as there are no functions $\bm{p}$ and $\bm{v}$ to which the operation in
\eqref{psi_N} can be applied.

We can write the solution to \eqref{SHE} in mild form as
the solution to the following stochastic integral equation:
\begin{equation}\label{mild}
	u(t\,,x) = 1 + \int_{(0,t)\times\R^d}\bm{p}_{t-s}(x-z)\sigma(u(s\,,z))\,
	\eta(\d s\,\d z);
\end{equation}
see Dalang \cite{Dalang1999} and Walsh \cite{Walsh}.

Since $\bm{p}_s\in\mathscr{S}(\R^d)$ for every $s>0$, we may apply Parseval's identity
to compute $\bm{p}_s*f$ and then integrate $[\exp(-\lambda s)\,\d s]$ in order to see that
for all $\lambda>0$ and $x\in\R^d$,
\begin{equation}\label{v:f}
	(\bm{v}_\lambda*f)(x) = \frac{2}{(2\pi)^d}\int_{\R^d}
	\frac{\e^{ix\cdot z} \hat{f}(z)}{2\lambda+\|z\|^2}\,\d z,
	\quad\text{whence}\quad
	(\bm{v}_\lambda*f)(0)=\Upsilon(\lambda),
\end{equation}
where $\Upsilon$ was defined in \eqref{Dalang}. Moreover, the inverse function $\Lambda$
to $\Upsilon$ --- see \eqref{Lambda:Upsilon} --- can be written in the following
alternative forms.
\[
	\Lambda(a) := \inf\left\{ \lambda>0:\ (\bm{v}_\lambda*f)(0) < a \right\}
	= \inf\left\{ \lambda>0:\ \Upsilon(\lambda)<a\right\}
	\qquad\text{for all $a>0$},
\]
where $\inf\varnothing:=\infty$. Since $\hat{f}(0)=f(\R^d)\in(0\,,\infty)$ and $\hat{f}$
is continuous, it follows from \eqref{v:f} that:
(a) $\Lambda(a)<\infty$ for all $a\neq0$ in all dimensions; and
(b) $\Lambda$ is
continuous and strictly decreasing on $(0\,,\infty)$.

\subsection{A Burkholder--Davis--Gundy inequality}
Suppose $L=\{L(s\,,z)\}_{s\ge0,z\in\R^d}$ is a predictable, space-time
random field. Then, the Walsh integral process $t\mapsto\int_{(0,t)\times\R^d}L\,\d\eta$
is a continuous, $L^2(\Omega)$-martingale with respect to the filtration
$\mathcal{F}$, and satisfies
\begin{equation}\label{BDG}
	\left\| \int_{\R_+\times\R^d} L\,\d\eta\right\|_k^2
	\le  4k\int_0^\infty
	\left( \| L(s\,,\bullet)\|_k * \| \widetilde{L(s\,,\bullet)}\|_k *f \right)(0)\,\d s,
\end{equation}
for every real number $k\ge2$ provided that the right-hand side of the above
inequality is finite at least when $k=2$, where
\[
	\tilde{\phi}(x) := \phi(-x)
\]
defines the (spatial) \emph{reflection} of every function $\phi:\R^d\to\R$.
Eq.\ \eqref{BDG} can be deduced from
the Burkholder--Davis--Gundy (BDG) inequality \cite{BDG}, using the fact that
the optimal constant in the BDG inequality is at most $\sqrt{4k}$ (see Carlen and
Kre\'e \cite{CarlenKree1991}). A derivation of \eqref{BDG} can be found in
Khoshnevisan \cite{KhCBMS} when $f$ is a function; see also \cite{Dalang1999}.
The present, more general, case where $f$ is a measure is proved by making small adjustment
to the latter argument. We skip the details.

\section{Proof of Theorem \ref{th:S}}\label{subsec:OccField}

Before we prove Theorem \ref{th:S} let us record two of its
ready consequences.

As a first application of Theorem \ref{th:S}, we may observe that it
implies \emph{a priori} statistical information about the (extended) random field $\mathcal{S}_{N,t}$.
For instance, Theorem \ref{th:S} and the stationarity of $u(t)$ \cite[Lemma 7.1]{CKNP} together imply
that
\[
	\E \left[ \mathcal{S}_{N,t}(\psi\,,g) \right] = 0
	\quad\text{and}\quad
	\Var\left[ N^{d/2} \mathcal{S}_{N,t}(\psi\,,g)\right] \lesssim\|\psi\|_{L^2(\R^d)}^2[\lip(g)]^2,
\]
uniformly for all $N,T>0$ and $t\in[0\,,T]$, and all $\psi\in L^2(\R^d)$
and $g\in\lip$. In this way,
we may conclude that
\[
	\lim_{N\to\infty} \left| \int_{\R^d} g(u(t\,,x))\psi_N(x)\,\d x
	- \E [g(u(t\,,0))]\int_{\R^d}\psi(x)\,\d x\right|=0
	\qquad\text{in $\bigcap_{k\ge2}L^k(\R^d)$},
\]
which is a generalization of the mean ergodic theorem
of Chen et al \cite{CKNP}, albeit in the special case that $f(\R^d)<\infty$. Once again,
we emphasize that the random variables inside the absolute value
are well defined whenever $\psi\in L^2(\R^d)$, though
$\int_{\R^d}\psi(x)\,\d x$ --- hence also $\int_{\R^d}g(u(t\,,x))\psi_N(x)\,\d x$ ---
might not converge absolutely.

As a second application of Theorem \ref{th:S}
we present the following tail-probability estimate.
It shows how the behavior of the spectral integral
$\Upsilon$ in \eqref{Dalang} affects the tails of
the distribution of the occupation field, uniformly
in the latter variable $N$.

\begin{lemma}\label{lem:tightness}
	For every $\varepsilon,\delta\in(0\,,1)$ and $t \in (0, T)$
	there exists $R_0=R_0(f\,,\varepsilon\,,\delta\,,T)>1$ such that
	\begin{equation}\label{eq:tightness:2}
		\sup_{N>0}\P\left\{ N^{d/2} \left| \mathcal{S}_{N,t}(\psi\,,g)\right| > \ell\right\}
		\le \exp\left\{ -\frac{\mathsf{a}(\varepsilon)\delta\log(\ell/\mathsf{B})}{2
		\Upsilon\left(\dfrac{1-\delta}{2T} \log(\ell/\mathsf{B})\right)} \right\}
		\qquad\text{for all $\ell>R_0\mathsf{B}$},
	\end{equation}
	where $\mathsf{B}:=\mathsf{A}(\varepsilon)\lip(g)\|\psi\|_{L^2(\R^d)}\sqrt{T}$,
	and both $\mathsf{a}(\varepsilon)$
	and $\mathsf{A}(\varepsilon)$ were defined in Theorem \ref{th:S}.
\end{lemma}

\begin{proof}
	For every $k\ge 2$, $t\in(0, T),\ell>0$, $\varepsilon\in(0\,,1)$, $\psi\in L^2(\R^d)$, and $g\in\lip$,
	\begin{equation}\label{eq:tightness:1}
		\sup_{N>0}\P\left\{ N^{d/2} \left| \mathcal{S}_{N,t}(\psi\,,g)\right| > \ell\right\}
		\le \exp\left\{-k\left[ \log\left(\frac{\ell}{\mathsf{B}}\right)
		-2T\Lambda \left( \frac{\mathsf{a}(\varepsilon)}{k}\right)
		- \tfrac12\log k\right]\right\}.
	\end{equation}
	The  inequality \eqref{eq:tightness:1}
	is an immediate consequence of Theorem \ref{th:S}
	and Chebyshev's inequality. We intend to apply \eqref{eq:tightness:1}
	with
	\[
		k =\frac{\mathsf{a}(\varepsilon)}{\Upsilon\left(
		\dfrac{1-\delta}{2T}\log(\ell/\mathsf{B})\right)},
	\]
	which is $\ge 2$ provided that $\ell/\mathsf{B}$ is sufficiently large
	since $\Upsilon$ vanishes at infinity.
	Since $\Lambda$ and $\Upsilon$ are inverses to one another, it
	then follows from \eqref{eq:tightness:1} that, as long as $\ell/\mathsf{B}$ is
	large enough,
	\begin{align*}
		&\sup_{N>0}\P\left\{ N^{d/2} \left| \mathcal{S}_{N,t}(\psi\,,g)\right| > \ell\right\}\\
		&\le \exp\left\{ -\frac{\mathsf{a}(\varepsilon)}{
			\Upsilon\left(\dfrac{1-\delta}{2T} \log(\ell/\mathsf{B})\right)}
			\left(\delta\log(\ell/\mathsf{B})-{ \frac{1}{2}}\log\left[
			\frac{\mathsf{a}(\varepsilon)}{\Upsilon\left(
			\dfrac{1-\delta}{2T}\log(\ell/\mathsf{B})\right)}\right]\right)\right\}.
	\end{align*}
	Next, observe from \eqref{Dalang} that
	\begin{equation}\label{LUL}
		\lambda\Upsilon(\lambda) \ge \frac{2}{(2\pi)^d}\int_{\|z\|<1}
		\frac{\hat{f}(z)\,\d z}{2+{\|z/\lambda\|}}
		\ge c:=\frac{{2}}{{3}(2\pi)^d}\int_{\|z\|<1}
		\hat{f}(z)\,\d z\qquad\text{whenever $\lambda>1$}.
	\end{equation}
	Because $\hat{f}$ is continuous
	and $\hat{f}(0)=f(\R^d)$,
	\eqref{f:finite} implies
	that $c$ is a strictly-positive real number.
	In particular,
	\[
		\delta\log(\ell/\mathsf{B})-{ \frac{1}{2}}\log\left[
		\frac{\mathsf{a}(\varepsilon)}{\Upsilon\left(
		\dfrac{1-\delta}{2T}\log(\ell/\mathsf{B})\right)}\right] \ge \frac\delta2\log(\ell/\mathsf{B}),
	\]
	as long as $\ell/\mathsf{B}$ is sufficiently large. Now we choose $R_0$ accordingly,
	all the time keeping careful track of the various parameter dependencies. This completes
	the proof.
\end{proof}

Equations \eqref{eq:tightness:1} and \eqref{eq:tightness:2} are essentially
equivalent. Moreover, they provide tail-probability estimates that
depend crucially on the rate at which $\Upsilon(\lambda)$ tends to zero as $\lambda\to\infty$.
Unfortunately, these tail-probability estimates are not particularly strong, though we have reason
to believe that they are not essentially improvable. For instance,  we might observe
from \eqref{LUL} that
\[
	\Upsilon(\lambda) \ge \frac{c}{\lambda}\qquad\text{for all $\lambda>1$},
\]
where $c>0$ does not depend on $\lambda$. Thus, it follows that
whenever $\ell/\mathsf{B}$ is sufficiently large,
\begin{equation}\label{eq:best-case}
	\exp\left\{ -\frac{\mathsf{a}(\varepsilon)\delta\log(\ell/\mathsf{B})}{
	2\Upsilon\left(\dfrac{1-\delta}{2T} \log(\ell/\mathsf{B})\right)} \right\}
	\ge \e^{ -
	\text{\rm const}\cdot \left| \log(\ell/\mathsf{B})\right|^2}
	\quad\text{whenever $\ell/\mathsf{B}\gg1$}.
\end{equation}
Since $|\log(\ell/\mathsf{B})|^2\to\infty$ slowly as $\ell/\mathsf{B}\to\infty$,
this shows that \eqref{eq:tightness:2} fails to produce fast decay of the tail probabilities:
The best rate we could hope for is given by the right-hand side
of \eqref{eq:best-case}.\footnote{%
	That rate can be achieved. For instance, suppose $f$ is bounded and continuous,
	as would happen for example if $\hat{f}\in L^1(\R^d)$. Then,
	$(\bm{v}_\lambda*f)(0)\le f(0)/\lambda$, and \eqref{v:f} shows that the
	right-hand side of \eqref{eq:tightness:2} is not greater than
	$\exp\{-\text{\rm const}\cdot|\log(\ell/\mathsf{B})|^2)$.}
And even the above bound is not a worst-possible case. For instance, suppose
$d=1$. In that case,
$\Upsilon(\lambda)\le f(\R)\pi^{-1}\int_{-\infty}^\infty(2\lambda+z^2)^{-1}\,\d z=
f(\R)/\sqrt{2\lambda}$
for every $\lambda>0$,
whence we obtain only\footnote{This rate is also unimprovable as can be seen by
inspecting the case $f=\delta_0$, for then $\Upsilon(\lambda)\propto\lambda^{-1/2}$
for all $\lambda>0$.}
\[
	\sup_{N>0}\P\left\{ \sqrt{N} \left| \mathcal{S}_{N,t}(\psi\,,g)\right| > \ell\right\}
	\le \exp\left\{ -\frac{\sqrt{T/2}\,\mathsf{a}(\varepsilon)\delta
	\left|\log(\ell/\mathsf{B})\right|^{3/2}}{
	f(\R)\sqrt{1-\delta}} \right\}
	\qquad\text{for all $\ell>R_0\mathsf{B}$}.
\]

We now return to Theorem \ref{th:S}, whose proof will require a preliminary lemma, and
follows the general ideas of Chen et al \cite{CKNP}. It has been proved in Chen
et al \cite[Theorem 6.4]{CKNP} that,
for each $t>0$ and $x\in\R^d$, the random variable
$u(t\,,x)$ is in the Gaussian Sobolev space $\mathbb{D}^{1,k}$  (see Nualart
\cite[Section 1.5]{Nualart}) for every
$k\ge2$, and that
\begin{equation}\label{D<p}
	\left\| D_{z,s}u(t\,,x) \right\|_k \lesssim \bm{p}_{t-s}(y-z),
\end{equation}
for all $t>0$ and $x\in\R^d$ and for a.e.\ $(s\,,z) \in (0\,,t)\times \R^d$,
where the implied constant depends only on
$(t\,,k)$. The following finds a numerical bound for that implied constant.

\begin{lemma}\label{lem:Du}
	For all real numbers $0<\varepsilon<1$,
	$T\ge t>0$, and $k\ge2$,
	and for every $x\in\R^d$,
	\begin{equation}  \label{D<p2}
		\left\| D_{s,z}u(t\,,x) \right\|_k \le
		\frac{{ 8}\left( |\sigma(0)|\vee\lip(\sigma)\right)
		\e^{2T\Lambda(\mathsf{a}(\varepsilon)/k)}} { {\varepsilon^{3/2}}}\,
		\bm{p}_{t-s}(x-z),
	\end{equation}
	valid for a.e.\ $(s\,,z)\in(0\,,t)\times\R^d$,where $\mathsf{a}(\varepsilon)$
	was defined in Theorem \ref{th:S}.
\end{lemma}

\begin{proof}
	Let $z_k$ denote the optimal constant in the BDG $L^k(\Omega)$-inequality
	for every real number $k\ge2$. Davis \cite{Davis1976}
	has evaluated $z_k$ in terms of the smallest root of a certain special function.
	Carlen and Kre\'e \cite{CarlenKree1991} have in turn shown that
	\[
		z_k \le 2\sqrt k\quad\text{for every $k\ge2$, and}\quad
		\sup_{\ell\ge2}\left( z_\ell/\sqrt \ell\right)=2.
	\]
	According to Chen et al \cite[(6.4)]{CKNP},
	\begin{equation}\label{|D|}
		\left\| D_{s,z}u(t\,,x) \right\|_k \le \frac{2C_{T,k}\e^{\lambda_0(t -s)}}{\sqrt{%
		1 - 2^{(d+2)/2} \left[z_k\lip  (\sigma)\right]^2
		\Upsilon(\lambda_0)}}
		\,\bm{p}_{t-s}(x-y),
	\end{equation}
	uniformly for all $0<t\le T$, $x\in\R^d$, and $k\ge2$,
	and for almost all $(s\,,z)\in(0\,,t)\times\R^d$. The constant $C_{T,k}$
	will be discussed shortly, and the preceding holds
	for all $\lambda_0$ large enough to ensure
	that $\Upsilon(\lambda_0) < 2^{-(d+2)/2}[ z_k\lip  (\sigma)]^{-2}$,
	equivalently $\lambda_0 > \Lambda(2^{-(d+2)/2}[ z_k\lip  (\sigma)]^{-2})$.
	Since $\Lambda$ is strictly decreasing
	and $z_k\le 2\sqrt{k}$ for all $k\ge1$, \eqref{|D|} holds with $z_k$
	replaced by $2\sqrt{k}$
	whenever $\lambda_0>\Lambda(1/\{k2^{(d+4)/2}[\lip (\sigma)]^2\}).$ Set
	\[
		\lambda_0:=\Lambda\left( \frac{%
		(1-\varepsilon)^2}{%
		k2^{(d+6)/2}[\lip(\sigma)]^2}\right),
	\]
	to obtain
	\begin{equation}\label{D:C}\begin{split}
		\left\| D_{s,z}u(t\,,x) \right\|_k &\le \frac{2C_{T,k}}{%
			{\sqrt{\varepsilon}}}\exp\left\{(t -s) \Lambda\left(
			\frac{(1-\varepsilon)^2}{k2^{(d+6)/2}[\lip(\sigma)]^2}\right)\right\}
			\,\bm{p}_{t-s}(x-y)\\
		&\le \frac{2C_{T,k}\e^{T\Lambda(\mathsf{a}(\varepsilon)/k)}}{%
			\sqrt{\varepsilon}}\,\bm{p}_{t-s}(x-y).
	\end{split}\end{equation}

	Now we address numerical bounds for the constant $C_{T,k}$. According to
	Chen et al \cite[Theorem 6.4]{CKNP}, we can select
	\[
		C_{T,k} := \sup_{t\in(0,T)}\sup_{x\in\R^d}\sup_{n\ge0}
		\|\sigma(u_n(t\,,x))\|_k,
	\]
	where
	\[
		u_{n+1}(t\,,x) = 1 + \int_{(0,t)\times\R^d}
		\bm{p}_{t-s}(x-y)\sigma(u_n(s\,,y))\,\eta(\d s\,\d y)
	\]
	denotes the $(n+1)$st-stage Picard iteration estimate of $u$ for all $n\ge1$,
	and $u_0(t\,,x) = 1$ for all $t\ge0$ and $x\in\R^d$. We warn that $u_n$ does not
	refer to the operation, defined in \eqref{psi_N}, that is applicable to a single spatial
	function on $\R^d$.

	Since $\sigma$ is Lipschitz continuous,
	\begin{equation}\label{C:le}
		C_{T,k} \le |\sigma(0)| + \lip(\sigma)\sup_{t\in(0,T)}\sup_{x\in\R^d}\sup_{n\ge0}
		\|u_n(t\,,x)\|_k.
	\end{equation}
	For every space-time random field $\Phi=\{\Phi(t\,,x)\}_{t\ge0,x\in\R^d}$
	and for all $k\ge2$ and $\beta>0$, define
	\[
		\mathcal{N}_{\beta,k}(\Phi) := \sup_{t\ge0}\sup_{x\in\R^d}
		\left( \e^{-\beta t}\|\Phi(t\,,x)\|_k\right).
	\]
	Our proof of \eqref{D<p} (see \cite[(5.9)]{CKNP}) hinges on the fact that
	\[
		\mathcal{N}_{\beta,k}(u_{n+1}) \le 1 + \left( |\sigma(0)| +
		\lip(\sigma)\mathcal{N}_{\beta,k}(u_n)\right)\sqrt{2k
		\Upsilon(\beta)},
	\]
	for all real numbers $k\ge2$ and $\beta>0$, and all integers $n\ge0$. Now suppose
	$\beta$ is so large that
	\[
		\Upsilon(\beta) \le
		\frac{(1-\varepsilon)^2}{2k\left\{ |\sigma(0)|\vee\lip(\sigma)\right\}^2}
		\quad\Leftrightarrow\quad
		\beta\ge\Lambda\left(\frac{(1-\varepsilon)^2}{
		2k\left\{ |\sigma(0)|\vee\lip(\sigma)\right\}^2}\right).
	\]
	For all values of $\beta$, we have
	$\mathcal{N}_{\beta,k}(u_{n+1}) \le { 2} + (1-\varepsilon)
	\mathcal{N}_{\beta,k}(u_n),$
	which yields the following upon iteration for every $n\ge0$:
	\[
		\mathcal{N}_\beta(u_{n+1}) \le { 2}\sum_{j=0}^n (1-\varepsilon)^j
		+ (1-\varepsilon)^{n+1}\mathcal{N}_\beta(u_0) \leq
		\frac{{ 2}(1-(1-\varepsilon)^{n+2})}{\varepsilon}.
	\]
	We choose the smallest such $\beta$, and unscramble the preceding to find that
	\[
		\sup_{x\in\R^d}\sup_{n\ge0}
		\|u_{n+1}(t\,,x)\|_k
		\le {\frac{{ 2}}{\varepsilon}}
		\exp\left\{t\Lambda\left(\frac{(1-\varepsilon)^2}{
		2k\left\{ |\sigma(0)|\vee\lip(\sigma)\right\}^2 }\right)\right\}
		\le { 2}\varepsilon^{-1}\e^{T\Lambda(\mathsf{a}(\varepsilon)/k)},
	\]
	valid for every real number $k\ge2$ and $t>0$, and all integers $n\ge0$.
	Since $u_0(t\,,x)=1$ and $\varepsilon\in(0\,,1)$,
	the right-most quantity in the previous display also bounds
	$\|u_0(t\,,x)\|_k=1$ from above. Therefore, \eqref{C:le} yields
	\[
		C_{T,k} \le |\sigma(0)| + \frac{{ 2}\lip(\sigma)
		\e^{T\Lambda(\mathsf{a}(\varepsilon)/k)}}{\varepsilon}
		\le \frac{{ 4}\left( |\sigma(0)|\vee\lip(\sigma)\right)
		\e^{T\Lambda(\mathsf{a}(\varepsilon)/k)}}{\varepsilon}.
	\]
	The lemma follows from this and \eqref{D:C}.
\end{proof}

In order to prove Theorem \ref{th:S}, we need the following technical result, which enables us to exchange the Malliavin derivative and integral.
Recall that for $g\in\lip$, Rademacher's theorem (see Federer
	\cite[Theorem 3.1.6]{Federer})
	ensures that $g$ has a weak derivative
	whose essential supremum is $\lip(g)$. Let $g'$ denote any measurable version of that derivative.

\begin{lemma}\label{lem:DS}
	Fix $t,N>0$, $\psi\in L^1(\R^d)\cap L^2(\R^d)$, and $g\in\lip$. Then,
	$\mathcal{S}_{N,t}(\psi\,,g)\in\mathbb{D}^{1,k}$
	for every $k\ge2$, and
	\[
		D_{s,z}\mathcal{S}_{N,t}(\psi\,,g) = \int_{\R^d}g'(u(t\,,x)) D_{s,z}u(t\,,x)
		\psi_N(x)\,\d x,
	\]
	for almost every $(s\,,z\,,\omega)\in\R_+\times\R^d\times\Omega$.
\end{lemma}

\begin{proof}

	Suppose first that $\psi\in C_c(\R^d)$. As it  has been mentioned before,
	we have shown in \cite{CKNP} that
	$D_{s,z}g(u(t\,,x))=g'(u(t\,,x))D_{s,z}u(t\,,x)$ a.s.\ for almost all
	$(s\,,z)\in\R_+\times\R^d$. We can approximate $\mathcal{S}_{N,t}(\psi\,,g)$ by discrete Riemann sums and then
    use the linearity and closability of the Malliavin derivative (see Nualart \cite[Proposition 1.2.1]{Nualart})
    to imply the result in this case. The general case follows
	from a density argument.
\end{proof}

Armed with Lemmas \ref{lem:Du} and \ref{lem:DS}, we proceed with a demonstration of Theorem \ref{th:S}.
\begin{proof}[Proof of Theorem \ref{th:S}]
	Define the random variable
	\[
		F := \int_{\R^d} g(u(t\,,x))\psi_N(x)\,\d x.
	\]
	By Lemma  \ref{lem:DS}, $F$ lies in the Gaussian Sobolev space
	$\mathbb{D}^{1,k}$ for every $k\ge2$, and
	\[
		D_{s,z}F = \int_{\R^d} g'(u(t\,,x)) D_{s,z}u(t\,,x) \psi_N(x)\,\d x,
	\]
	almost surely for a.e.\ $(s\,,z)\in\R_+\times\R^d$.
	Apply the Clark--Ocone formula, in the form given in \cite[Proposition 4.3]{CKNP}, in order to see that
	\[
		F-\E(F)
		= \int_{(0,t)\times\R^d}\eta(\d s\,\d z)\int_{\R^d}\psi_N(x)\,\d x\
		\E\left( g'(u(t\,,x)) D_{s,z}u(t\,,x) \mid \mathcal{F}_s\right),
	\]
	almost surely.  To simplify the notation,
	define
	\[
		L(s\,,z) := \int_{\R^d}\psi_N(x)
		\E\left( g'(u(t\,,x)) D_{s,z}u(t\,,x) \mid \mathcal{F}_s\right)\d x,
	\]
	so that the preceding can be restated as $F-\E(F)=\int_{(0,t)\times\R^d}L\,\d\eta$.
	Thus, the BDG inequality \eqref{BDG} implies the following Poincar\'e inequality:
	\[
		\| F-\E(F)\|_k
		\le 2\sqrt{k\int_0^t \left( \|L(s\,,\bullet)\|_k *
		\widetilde{\| L(s\,,\bullet)\|}_k * f\right)(0)\,\d s}.
	\]
	Since $\|g'\|_{L^\infty(\R^d)} =\lip(g)$, it follows from the conditional
	Jensen inequality that
	\begin{align*}
		\|L(s\,,z)\|_k &\le \lip(g)\int_{\R^d}|\psi_N(x)| \left\| D_{s,z}u(t\,,x)\right\|_k\,\d x\\
		&\le \frac{{ 8}\lip(g)\left( |\sigma(0)|\vee\lip(\sigma)\right)
		\e^{2T\Lambda(\mathsf{a}(\varepsilon)/k)}}{{\varepsilon^{3/2}}}
		\left( |\psi_N|*\bm{p}_{t-s}\right)(z);
	\end{align*}
	see Lemma \ref{lem:Du} for the last line. Therefore, we can combine the above
	bounds with the semigroup property of the heat kernel
	in order to reach the following conclusion:
	\[
		\| F-\E(F)\|_k\le\frac{{ 16}\lip(g)\left( |\sigma(0)|\vee\lip(\sigma)\right)
		\e^{2T\Lambda(\mathsf{a}(\varepsilon)/k)}}{ {\varepsilon^{3/2}}}
		\sqrt{k\int_0^t\left( |\psi_N| * |\tilde{\psi}_N| *
		\bm{p}_{2(t-s)} * f\right)(0)\,\d s
		}.
	\]
	In accord with Young's inequality for convolutions,
	$|\psi_N|*|\tilde{\psi}_N|\le\|\psi_N\|_{L^2(\R^d)}^2
	= N^{-d}\|\psi\|_{L^2(\R^d)}^2$ a.e. This implies that
	\[
		\left( |\psi_N| * |\tilde{\psi}_N| *\bm{p}_{2(t-s)} * f\right)(0)
		\le N^{-d}\|\psi\|_{L^2(\R^d)}^2\int_{\R^d}\left(
		\bm{p}_{2(t-s)}*f\right)(x)\,\d x = N^{-d}\|\psi\|_{L^2(\R^d)}^2f(\R^d),
	\]
	and concludes the proof.
\end{proof}

\section{Short-range dependence}

Let $U:=\{U(x)\}_{x\in\R^d}$ be a stationary random field such that
$\E(|U(0)|^2)<\infty$. Recall that $U$ is  said to be \emph{short-range
dependent} if
\[
	\int_{\R^d}\left| \Cov\left[ U(x)\,,U(0)\right]\right|\d x<\infty.
\]
It is a well-known observation that when $U$ is short-range dependent, the
non-random quantity
$\chi:=\int_{\R^d}\Cov[U(x)\,,U(0)]\,\d x$ is finite and absolutely convergent, and
\begin{align*}
	\Var\left( \frac{1}{N^{d/2}}\int_{[0,N]^d}U(x)\,\d x\right)
		&= \frac{1}{N^d}\int_{[0,N]^d}\d x\int_{[0,N]^d}\d y\
		\Cov\left[ U(x-y)\,,U(0)\right]\\
	&\to\chi\qquad\text{as $N\to\infty$}.
\end{align*}

\subsection{Asymptotics for the variance}
Among other things, in this section we will prove as a direct consequence of
\eqref{f:finite} that, whenever $g\in\lip$,
the stationary and square-integrable random field $g(u(t\,,\bullet))$ is short-range dependent.
We explore some consequences of this short-range dependence as well.
\begin{lemma}\label{lem:Cov:int}
	For every $t,T\ge0$ and $g,G\in\lip$,
	\[
		\int_{\R^d}\left| \Cov\left[ g(u(t\,,x)) \,, G(u(T\,,0))\right] \right|\d x
		<\infty.
	\]
	Consequently, $g(u(t\,,\bullet))$ is short-range dependent for every
	$t\ge0$ and $g\in\lip$.
\end{lemma}

Before we prove Lemma \ref{lem:Cov:int}, we digress to talk about the role of
Lemma \ref{lem:Cov:int} in our discussion.

In accord with Lemma \ref{lem:Cov:int},
\begin{equation}\label{def:B_t,T}
	\mathbf{B}_{t,T}(g\,,G) := \int_{\R^d}\Cov\left[ g(u(t\,,x)) \,, G(u(T\,,0))\right]\d x
\end{equation}
is a real number for every $t\ge0$ and $g,G\in\lip$.

We have already mentioned the fact that every $u(t)$ is spatially stationary.
It is proved in Chen et al \cite{CKNP} that in fact $u$ is spatially stationary; that is,
the infinite-dimensional process $\{u(t\,,x+y);\, t\ge0,x\in\R^d\}$
has the same law as $\{u(t\,,x);\, t\ge0,x\in\R^d\}$ for every $y\in\R^d$. This extended
form of stationarity readily implies the following:
\begin{compactenum}
	\item  The form $\mathbf{B}_{t,T}:\lip^2\to\R$
		is bilinear for every $t,T\ge0$.
	\item The  form $\mathbf{B}:(t\,,g)\times(T\,,G)\in\R_+^2\times\lip^2\to
		\mathbf{B}_{t,T}(g\,,G)\in\R$
		is symmetric and non-negative definite.
\end{compactenum}
As a consequence, general theory ensures the existence of
a centered Gaussian random field
\begin{equation}\label{gamma}
	\Gamma := \left\{ \Gamma_t(\psi\,,g);\, t\ge0,\,\psi\in L^2(\R^d),\, g\in\lip\right\},
\end{equation}
whose covariance form is given by
\begin{equation*}\label{Cov:gamma}
	\Cov\left[ \Gamma_t(\psi\,,g) ~,~ \Gamma_T(\Psi\,,G)\right] =
	\<\psi\,,\Psi\>_{L^2(\R^d)}\cdot\mathbf{B}_{t,T}(g\,,G),
\end{equation*}
for every $t,T\ge0$, $g,G\in\lip$, and $\psi,\Psi\in L^2(\R^d)$.
The bilinear form
that appeared earlier in Theorem \ref{th:1.1} is defined in terms of
$\mathbf{B}_{t,T}$ as follow: For every $t\ge0$ and $(g\,,G)\in\lip\times\lip$,
\begin{equation}\label{def:B_t}
	\mathbf{B}_t(g\,,G) := \mathbf{B}_{t,t}(g\,,G)=
	\int_{\R^d}\Cov\left[ g(u(t\,,x)) \,, G(u(t\,,0))\right]\d x,
\end{equation}
and is the covariance of the centered Gaussian process $\Gamma_t(\psi\,,\bullet)$
for every fixed $t\ge0$ and $\psi\in L^2(\R^d)$ such that $\|\psi\|_{L^2(\R^d)}=1$.

We can now verify Lemma \ref{lem:Cov:int}.

\begin{proof}[Proof of Lemma \ref{lem:Cov:int}]
	We showed in the course of the proof of Theorem \ref{th:S} that
	for all $t\ge0$ and $x\in\R^d$,	the following Clark--Ocone formula holds a.s.:
	\[
		g(u(t\,,x)) - \E[g(u(t\,,x))] = \int_{(0,t)\times\R^d}
		\E\left( g'(u(t\,,x)) D_{s,z}u(t\,,x) \mid \mathcal{F}_s\right)\eta(\d s\,\d z).
	\]
	Of course, a similar expression holds when we replace
	$(g\,,t\,,x)$ by $(G\,,T\,,0)$ everywhere as well.
	For almost every $s>0$ and $z\in\R^d$, the following random variables
	are well defined:
	\begin{gather*}
		\ell_s(z) := \E\left( g'(u(t\,,x)) D_{s,z}u(t\,,x) \mid \mathcal{F}_s\right)  \quad {\rm and} \quad
		L_s(y) := \E\left( G'(u(T\,,0)) D_{s,y}u(T\,,0) \mid \mathcal{F}_s\right),
	\end{gather*}
	and in fact define $L^2(\Omega)$-continuous  ---
	whence also Lebesgue measurable --- processes indexed by $(s\,,z)$;
	see Chen et al \cite{CKNP}.
	Set  $\mathcal{W}_s(y\,,z):=\E[\ell_s(z) L_s(y)]$.
	It follows from the Walsh isometry for stochastic integrals that
	\begin{equation}\label{Cov:W}
		\Cov\left[ g(u(t\,,x))\,, G(u(T\,,0))\right]  =
		\int_0^{t\wedge T}\d s\int_{\R^d}
		\left(\mathcal{W}_s(y\,,\bullet)*f\right)(y)\,\d y.
	\end{equation}
	The term $t\wedge T$ appears here because of the fact that
	if $F\in\mathbb{D}^{1,2}$ is measurable with respect to $\mathcal{F}_t$
	for some $t\ge0$, then $D_{s,z}F=0$ when $s\ge t$;
	see Nualart \cite{Nualart}.

	Since $g'$ and $G'$ are respectively essentially bounded by
	$\lip(g)$ and $\lip(G)$,
	we first apply  the Cauchy-Schwarz inequality and then the conditional Jensen's inequality, in this order, to find that
	\begin{align*}
		\left| \mathcal{W}_s(y\,,z)\right| &\le\left\| \ell_s(z)\right\|_2
			\left\| L_s(y)\right\|_2\\
		&\le\lip(g)\lip(G)\left\| D_{s,z}u(t\,,x)\right\|_2
			\left\| D_{s,y}u(T\,,0)\right\|_2.
	\end{align*}
	Apply Lemma \ref{lem:Du} with $k=2$ in order to find that
	\[
		\left| \mathcal{W}_s(y\,,z)\right|
		\le K\bm{p}_{t-s}(x-z)\bm{p}_{T-s}(y),
	\]
	where the constant $K$ depends only on $(f\,,g\,,G\,,\sigma\,,t, T)$.
	It follows from this, the semigroup property of the heat kernel,
	and \eqref{Cov:W} that
	\[
		\left| \Cov\left[ g(u(t\,,x))\,, G(u(T\,,0))\right] \right| \le
		K \int_0^{t\wedge T} \left( \bm{p}_{T+t-2s}*f\right)(x)\,\d s,
	\]
	whence
	\[
		\int_{\R^d}\left| \Cov\left[ g(u(t\,,x))\,, G(u(T\,,0))\right] \right|\d x
		\le K(t\wedge T)f(\R^d)<\infty,
	\]
	thanks to \eqref{f:finite}.
\end{proof}

Lemma \ref{lem:Cov:int}, the discussion at the beginning of this section
and \eqref{def:B_t} together imply immediately that
\[
	\lim_{N\to\infty}
	\Var\left( \frac{1}{N^{d/2}}\int_{[0,N]^d} g(u(t\,,x))\,\d x\right)
	=\mathbf{B}_t(g\,,g),
\]
for all $t\ge0$ and $g\in\lip$. The following result generalizes this fact
to an asymptotic behavior of the covariance form of the normalized
occupation field.

\begin{proposition}\label{pr:Cov:asymp}
	For every $t\ge0$, $\psi,\Psi\in L^2(\R^d)$, and $g,G\in\lip$,
	\[
		\lim_{N\to\infty}
		\Cov\left[ N^{d/2} \mathcal{S}_{N,t}(\psi\,,g)\,,
		N^{d/2} \mathcal{S}_{N,t}(\Psi\,,G)\right] =
		\< \psi\,,\Psi \>_{L^2(\R^d)}
		\cdot\mathbf{B}_t(g\,,G).
	\]
\end{proposition}

\begin{proof}
	First, consider the case that $\psi,\Psi\in {L^1(\R^d)\cap L^2(\R^d)}$. In that case,
	\[
		\Cov\left[ \mathcal{S}_{N,t}(\psi\,,g)\,, \mathcal{S}_{N,t}(\Psi\,,G)\right]
		= \int_{\R^d}\psi_N(x)\,\d x\int_{\R^d}\Psi_N(y)\,\d y\
		\Cov\left[ g(u(t\,,x-y)) \,, G(u(t\,,0))\right].
	\]
	Define
	\[
		\phi(z) := \Cov\left[ g(u(t\,,z)) \,, G(u(t\,,0))\right]\qquad\text{for all $z\in\R^d$},
	\]
	in order to deduce the formula
	\begin{equation}\label{Cov:psi:Psi}
		\Cov\left[ \mathcal{S}_{N,t}(\psi\,,g)\,, \mathcal{S}_{N,t}(\Psi\,,G)\right]
		= \left(\psi_N*\tilde{\Psi}_N*\phi\right)(0).
	\end{equation}
	Lemma \ref{lem:Cov:int} ensures that $\phi\in L^1(\R^d)$; and because
	$g,G\in\lip$ and $u$ is (jointly) continuous in $L^2(\Omega)$ --- see
	Dalang \cite{Dalang1999} --- both $\phi$ and its
	Fourier transform $\hat\phi$ are continuous and bounded.
	Parseval's identity applies and tells us that we can recast \eqref{Cov:psi:Psi}
	as follows:
	\[
		\Cov\left[ \mathcal{S}_{N,t}(\psi\,,g)\,, \mathcal{S}_{N,t}(\Psi\,,G)\right]
		=\frac{1}{(2\pi)^d}\int_{\R^d}\hat{\psi}_N(z)\overline{\hat{\Psi}_N(z)}
		\hat{\phi}(z)\,\d z
		= \frac{1}{(2\pi N)^d}\int_{\R^d}\hat{\psi}(w)\overline{\hat{\Psi}(w)}
		\hat{\phi}(w/N)\,\d w,
	\]
	after a change of variables $[w=Nz]$. Let $N\to\infty$, appeal to the continuity and
	boundedness of $\hat\phi$ as well as the dominated convergence theorem in order
	to find that
	\begin{equation}\label{Cov:N}
		\Cov\left[ N^{d/2} \mathcal{S}_{N,t}(\psi\,,g)\,, N^{d/2} \mathcal{S}_{N,t}(\Psi\,,G)\right]
		\to \frac{\hat{\phi}(0)}{(2\pi)^d}\int_{\R^d}
		\hat{\psi}(w)\overline{\hat{\Psi}(w)}\,\d w
		\qquad\text{as $N\to\infty$}.
	\end{equation}
	This is another way to state Proposition \ref{pr:Cov:asymp} in the special case
	that $\psi,\Psi\in {L^1(\R^d)\cap L^2(\R^d)}$. Now, Theorem \ref{th:S} ensures
	that the quantity on the left-hand side of \eqref{Cov:N} densely defines a
	continuous functional of $(\psi\,,\Psi)\in L^2(\R^d)\times L^2(\R^d)$,
	uniformly in $N>0$. And the right-hand side is also such a continuous functional
	thanks to the Cauchy-Schwarz inequality. Therefore, \eqref{Cov:N} and a
	standard density argument
	together imply the proposition in its full generality.
\end{proof}

\subsection{Comments on non-degeneracy}

The conclusion of Proposition \ref{pr:Cov:asymp} is consistent with the $N^{d/2}$
scaling of the CLT for the occupation field $\mathcal{S}[t]$
in Theorem \ref{th:1.1}. Moreover,
we see that the asymptotic covariance of the occupation field, properly normalized,
is a multiple of the form $\mathbf{B}_t(g\,,G)$. Thus,
it would be nice to know conditions under which the rate $N^{d/2}$ of the convergence in the
CLT of Theorem \ref{th:1.1} is non-degenerate.
We can recast this question by asking the following:
\begin{quotation}
	\emph{Given a number $t\ge0$,
	is $\mathbf{B}_t(g\,,g) > 0$ for some $g\in\lip$?}
\end{quotation}
This is equivalent to asking
whether the limiting Gaussian process $\Gamma_t$ of Proposition \ref{pr:1}
is non degenerate for given value of $t\ge0$. Since $u(0)\equiv1$, $\mathbf{B}_0(g\,,g)=0$ for all $g\in\lip$.
Thus, the question is interesting only when $t>0$. Additionally, the question is interesting
only when $\sigma(1)\neq0$, for $u(t)\equiv1$ otherwise, which renders
$\Gamma_t$ degenerate for all $t\ge0$.

The following lemma gives a partial answer to the mentioned non-degeneracy question.

\begin{proposition}\label{prop:ND}
	Suppose $\sigma$ satisfies one of the following conditions:
	\begin{compactenum}
		\item Either there exists $c_0>0$ such that $\sigma(w)\ge c_0$
			for all $w>0$ or $\sigma(w)\le -c_0$ for all $w>0$; or
		\item $\sigma(0)=0$, and there exists $c_1>0$ such that
			either $\sigma(w)\ge c_1w$ for all $w>0$ or
			$\sigma(w)\le -c_1w$ for all $w>0$.
        \item $\sigma(1)\ne 0$, $\sigma(0)=0$, and either $\sigma(x)$ or $-\sigma(x)$ is nonnegative for all $x>0$.
	\end{compactenum}
	Then, there exists $g\in\lip$ such
	that $\mathbf{B}_t(g\,,g)>0$ for every $t>0$.
	Moreover,  either condition 1 or 2 implies the existence of a constant $c>0$ such that
	$\mathbf{B}_t(g\,,g) \ge c t f(\R^d)>0;$
	and under condition 3, there exist a constant $\delta\in (0\,,t)$ and $R>0$ such that
	\begin{equation}\label{Bf}
		\mathbf{B}_t(g\,,g) \ge 2^{-(d+2)/2} \sigma^2(1)\delta f\left([-R\,,R]^d\right)>0.
	\end{equation}
\end{proposition}

\begin{proof}
	Throughout the proof, we consider only the Lipschitz-continuous
	function
	\[
		g(w) =w \quad\text{for all $w\in\R$},
	\]
	and choose and fix an arbitrary number $t>0$. In order to simplify the exposition,
	we work in the case that $f$ is additionally a \emph{function}; the general case
	that $f$ is a measure works in a similar way though the notation is slightly messier.
	Therefore, we omit the proof of the general case.

	If condition 1 of the proposition holds, then \eqref{mild}, the semigroup properties
	of the heat kernel, the basic properties of
	the Walsh stochastic integral, and the spatial stationarity of $u(s)$ together imply that
	\begin{align*}
		&\Cov\left[ g(u(t\,,x))\,,g(u(t\,,0))\right] =
			\E\left[ u(t\,,x)u(t\,,0)\right]-1\\
		&\hskip1in= \int_0^t\d s\int_{\R^d}\d y\int_{\R^d}\d w\
			\bm{p}_{t-s}(x-y+w)\bm{p}_{t-s}(y)
			\E\left[\sigma(u(s\,,w))\sigma(u(s\,,0))\right]f(w)\\
		&\hskip1in\ge c_0^2 \int_0^t\d s\int_{\R^d}\d w\
			\bm{p}_{2(t-s)}(x+w)f(w)
			= c_0^2\int_0^t\left( \bm{p}_{2s}*f\right)(x)\,\d s,
	\end{align*}
	which is strictly positive thanks to \eqref{f:finite}. (The final inequality holds also
	when $f$ is a measure, and for similar reasons.) Because $u(t)$ is continuous in
	$L^2(\Omega)$ (see Dalang \cite{Dalang1999}), the left-most quantity defines
	a continuous function of $x$. Therefore,
	we may integrate $[\d x]$ to see that
	$\mathbf{B}_t(g\,,g)>0$ for the present choice of $g$.
	This proves that condition 1 implies the strict positivity of $\mathbf{B}_t(g\,,g)$.

	Next suppose condition 2 holds. According to the weak comparison theorem of
	Chen and Huang (see \cite[Corollary 1.4]{CH19}),
	$\P\{u(t\,,x)\ge0\}=1$ for every $x\in\R^d$. Thus,
	a similar computation as above yields
	\begin{align*}
		\E\left[ u(t\,,x)u(t\,,0)\right]
			&= 1 + \int_0^t\d s\int_{\R^d}\d y\int_{\R^d}\d w\
			\bm{p}_{t-s}(x-y+w)\bm{p}_{t-s}(y)
			\E\left[\sigma(u(s\,,w))\sigma(u(s\,,0))\right]f(w)\\
		&\ge 1 + c_1^2 \int_0^t\d s\int_{\R^d}\d w\
			\bm{p}_{2(t-s)}(x+w)f(w)\E\left[u(s\,,w)u(s\,,0)\right].
	\end{align*}
	The asserted non-negativity of $u(s)$ implies now that $\E[u(s\,,x)u(s\,,0)]\ge1$
	for every $x\in\R^d$. We enter this bound back into the right-hand side of the above
	in order to see that
	\[
		\Cov\left[ g(u(t\,,x))\,,g(u(t\,,0))\right]
		\ge c_1^2 \int_0^t\d s\int_{\R^d}\d w\
		\bm{p}_{2(t-s)}(x+w)f(w)
		= c_1^2\int_0^t\left( \bm{p}_{2s}*f\right)(x)\,\d s.
	\]
	Now proceed as we did under condition 1 to deduce that $\mathbf{B}_t(g\,,g)>0$ under condition 2.

	Finally, suppose condition 3 holds. Set ${h}(s\,,w):=\E[\sigma(u(s\,,w))\sigma(u(s\,,0))]$
	and observe that $(s\,,w)\mapsto {h}(s\,,w)$ is continuous
	for all $s\ge 0$ and $w\in\R^d$. (This follows from the continuity of $u$ in $L^2(\Omega)$.)
	Because ${h}(0\,,w)\equiv \sigma^2(1) >0$ for all $w\in\R^d$,
	there exist $\delta\in (0\,,t)$ and $R>0$ such that
	\begin{equation}\label{E:inf_g}
		\inf_{(s,w)\in [0,\delta]\times[-R,R]^d} {h}(s,w) \ge \sigma^2(1)/2.
	\end{equation}
	Condition 3 and the fact that $u(t\,,x)\ge 0$ a.s. (see \cite{CH19}),  together imply that
	$g(s\,,w)\ge 0$. It follows from \eqref{E:inf_g} that
	\[
		\Cov\left[ g(u(t\,,x))\,,g(u(t\,,0))\right]
		\ge \frac{\sigma^2(1)}{2} \int_0^\delta \d s\int_{[-R,R]^d}\d w\
		\bm{p}_{2(t-s)}(x+w)f(w).
	\]
	Integrate $[\d x]$ to deduce the inequality in \eqref{Bf}. Thus, it remains to prove
	that $f([-R\,,R]^d)>0$. We will prove the following more general fact:
	\begin{equation}\label{f>0}
		f\left([0\,,r]^d\right)>0\qquad\text{for every $r>0$}.
	\end{equation}

	Define
	\[
		I_r(x) := r^{-d}\bm{1}_{[0,r]^d}(x)\qquad\text{for every $r>0$ and $x\in\R^d$}.
	\]
	As we observed in \cite[(3.17)]{CKNP}, for every $r>0$,
	\begin{equation}\label{III}
		(2r)^{-d}\bm{1}_{[0,r/2]^d} \le I_r*\tilde{I}_r \le r^{-d}\bm{1}_{[0,r]^d}
		\qquad\text{on $\R^d$},
	\end{equation}
	where $\tilde{h}(x) := h(-x)$. Thus,
	\[
		f\left({ x + [0\,,r/2]^d}\right) =\int\bm{1}_{[0,r/2]^d}(w-x)\, f(\d w)
		\le (2r)^d \left(I_r*\tilde{I}_r*f\right)(x),
	\]
	for every $r>0$ and $x\in\R^d$. Since $I_r*\tilde{I}_r*f$ is continuous and positive
	definite, it is maximized at $x=0$. Thus, a second application of \eqref{III} yields
	\[
		\sup_{x\in\R^d}
		f\left({x+ [0\,,r/2]^d}\right) \le (2r)^d \left(I_r*\tilde{I}_r*f\right)(0)
		\le 2^d\int\bm{1}_{[0,r]^d}(w)\,f(\d w)=2^df\left([0\,,r]^d\right).
	\]
	Now suppose to the contrary that $f([0\,,r]^d)=0$ for some $r>0$. If so, then
	the preceding implies that
	\[
		f\left( {j+ [0\,,r/2]^d}\right)=0
		\qquad\text{for all $j\in{\frac{r}{2}}\mathbb{Z}^d$}.
	\]
	Sum the above quantity over all $j\in{\frac{r}{2}}\mathbb{Z}^d$ in order to deduce that $f(\R^d)=0$,
	thus contradicting \eqref{f:finite}. This verifies \eqref{f>0} and completes the
	proof.
\end{proof}

\subsection{Proof of necessity in Theorem \ref{th:main}}\label{subsec:T1:2}
We are ready to prove the easy half of Theorem \ref{th:main}. Namely,
we plan to prove that if $\sigma$ is a non-zero constant --- say $\sigma\equiv c_0\neq0$ ---
and the central limit theorem \eqref{CLT} holds
for every $t>0$ and $g\in\lip$, then $f(\R^d)<\infty$.

Set  $g(w)=w$ for all $w\in\R$ and
$S_N := N^{-d/2}\int_{[0,N]^d} u(1\,,x)\,\d x$
for all $N>0$.
Since $S_N$ has a normal distribution with mean $\E[u(1\,,0)]=1$, \eqref{CLT}
implies that
\begin{equation}\label{Var(X)}
	\lim_{N\to\infty}\Var\left( S_N\right)
	=\Var(X)<\infty.
\end{equation}
Thanks to stationarity, \eqref{mild}, and the $L^2(\Omega)$-isometry properties of Walsh stochastic
integrals,
\begin{align*}
	\Var(S_N)&=\frac{1}{N^d}\int_{[0,N]^d}\d x\int_{[0,N]^d}\d y\
		\Cov[u(1\,,x)\,,u(1\,,y)]
		\\
	&\to \int_{\R^d}\Cov[u(1\,,z)\,,u(1\,,0)]\,\d z
		= c_0^2\int_0^1\d s \int_{\R^d}\d z\
		(\bm{p}_{2s}*f)(z)=c_0^2  f(\R^d),
\end{align*}
as $N\to\infty$. Thus, we can conclude from \eqref{Var(X)},
that $f(\R^d)<\infty$.\qed

\section{Asymptotic independence}

The primary goal of this section is to prove that $\mathcal{S}_{N,t}(\psi\,,g)$
has good ``independence properties,''
as $\psi$ ranges over a sufficiently-large portion of $L^2(\R^d)$.
Before we begin that discussion, let us recall a notion of asymptotic independence
that is relevant to us.

\begin{definition}\label{def:AI}
	Choose and fix an integer $m\ge1$, and let $X=\{X_{j,N};\, 1\le j\le m,\, N>0\}$.
	We say that $X$ has \emph{asymptotic independence} when
	\[
		\lim_{N\to\infty}\left|
		\E\left[\e^{i\sum_{j=1}^m z_j X_{j,N}}\right] - \prod_{j=1}^m\E
		\left[ \e^{iz_j X_{j,N}} \right]
		\right| =0\qquad\text{for every $z_1,\ldots,z_m\in\R$}.
	\]
\end{definition}

Suppose $X$ has asymptotic independence, and (as $N\to\infty$) $X_{j,N}$
converges weakly to a probability measure $\mu_j$ for every $j=1,\ldots,m$.
Then it follows immediately from Definition \ref{def:AI} that
$( X_{1,N}\,,\ldots, X_{m,N}) $ converges in distribution to
$\mu_1\times\cdots\times\mu_m$ as $N\to\infty$. This property is the main
motivation behind the definition of asymptotic independence.

\begin{theorem}\label{th:AI}
	Choose and fix $t>0$ and $g\in\lip$, and suppose that
	$\phi,\psi\in L^2(\R^d)$ both have compact support.
	Then,
	\begin{align}\label{eq:AI}
	\begin{aligned}
		&\left| \Cov\left[ \exp\left(i N^{d/2}\mathcal{S}_{N,t}(\psi,g) \right)
		~,~ \exp\left(-iN^{d/2}\mathcal{S}_{N,t}(\phi,g)\right) \right]\right|  \\
		&\hskip2in\lesssim\int_0^t\d s\int_{\R^d}\d \eta \
			\left(\bm{p}_{2s}*f\right)(\eta)\left(|\phi|*\tilde{|\psi|}\right)\left(\frac{\eta}{N}\right),
	\end{aligned}
	\end{align}
	uniformly for all $N>0$, where the implied constant does not depend on $(\psi\,,\phi\,,N)$.
	Consequently, if the intersection of the supports of $\phi$ and $\psi$
	is Lebesgue-null, then
	$N^{d/2}\mathcal{S}_{N,t}(\psi\,,g)$
	and $N^{d/2}\mathcal{S}_{N,t}(\phi\,,g)$ are asymptotically independent
	as $N\to\infty$.
\end{theorem}

\begin{proof}
	In order to simplify the typesetting define
	\[
		\Phi := \exp\left( i N^{d/2}\mathcal{S}_{N,t}(\psi\,,g)\right),\quad
		\Psi := \exp\left( -iN^{d/2} \mathcal{S}_{N,t}(\phi\,,g)\right).
	\]
	According to Lemma \ref{lem:DS}, the Clark--Ocone formula (see Chen et al
	\cite{CKNP}) and the chain rule of Malliavin calculus (see Nualart \cite{Nualart}),
	$\Phi,\Psi\in\mathbb{D}^{1,k}$ for every $k\ge2$,
	\begin{align*}
		\Phi-\E(\Phi) &=iN^{d/2}\int_{(0,t)\times\R^d}
			\E\left(\left.\Phi\int_{\R^d}g'(u(t\,,x)) D_{s,z}u(t\,,x)\psi_N(x)
			\,\d x\ \right|\,\mathcal{F}_s\right)\eta(\d s\,\d z),
			\text{ and}\\
		\Psi-\E(\Psi) &=-iN^{d/2}\int_{(0,t)\times\R^d}
			\E\left(\left.\Psi\int_{\R^d}g'(u(t\,,x)) D_{s,z}u(t\,,x)\phi_N(x)
			\,\d x\ \right|\,\mathcal{F}_s\right)\eta(\d s\,\d z),
	\end{align*}
	almost surely. In order to further simplify the exposition and the notation, suppose for now
	that the correlation $f$ is a function. In that case,  Walsh isometry for
	stochastic integrals ensures that
	\begin{align*}
		&\Cov(\Phi\,,\Psi) = \E\left( \left[ \Phi-\E(\Phi)\right]\cdot
			\overline{\left[
			\Psi -\E(\Psi) \right]}\right) =
			-N^{d}\E\int_0^t\d s\int_{\R^d}\d y\int_{\R^d}\d z\
			f(y-z)\\
		&\times\E\left(\left.\Phi\int_{\R^d}g'(u(t\,,a)) D_{s,y}u(t\,,a)\psi_N(a)
			\,\d a\ \right|\,\mathcal{F}_s\right)
			\E\left(\left.\bar\Psi\int_{\R^d}g'(u(t\,,b)) D_{s,z}u(t\,,b)\phi_N(b)
			\,\d b\ \right|\,\mathcal{F}_s\right).
	\end{align*}
	In particular, we may use the Cauchy-Schwarz inequality, the conditional Jensen's inequality, and the fact that
	$|\Psi|\vee|\Phi|\le1$ in order to see that
	\[
		\left|\Cov(\Phi\,,\Psi)\right|
		\le N^d[\lip(g)]^2\int_0^t\d s\int_{\R^d}\d y\int_{\R^d}\d z\
		f(y-z) \mathcal{A}\mathcal{B},
	\]
	where $\mathcal{A} := \int_{\R^d}\| D_{s,y}u(t\,,a)\|_2
	|\psi_N(a) |\,\d a$ and $\mathcal{B} := \int_{\R^d}
	\| D_{s,z}u(t\,,b)\|_2 |\phi_N(b)|\,\d b.$
	In accord with Lemma \ref{lem:Du},
	\[
		\mathcal{A} \lesssim \left( \bm{p}_{t-s}*|\psi_N|\right)(y)
		\quad\text{and}\quad
		\mathcal{B} \lesssim \left( \bm{p}_{t-s}*|\phi_N|\right)(z),
	\]
	for almost all $0<s<t$ and $y,z\in\R^d$. We emphasize that the implied
	constants do not depend on any of the interesting variables here (see
	Lemma \ref{lem:Du} for numerical bounds on these constants.)
	Consequently,
	\begin{align*}
		\left|\Cov(\Phi\,,\Psi)\right| \lesssim
		N^d\int_0^t\left( \bm{p}_{2s}*|\psi_N| * |\tilde{\phi}_N|*f\right)(0)\,\d s.
	\end{align*}
	Once again, the implied constants are harmless. Even though we have obtained this
	inequality under the additional hypothesis that $f$ is a function, it is possible to check
	that the very same inequality holds more generally when $f$ is a measure.

	Now we unscramble the convolutions in order to see that
	\begin{align*}
		\left|\Cov(\Phi\,,\Psi)\right| &\lesssim N^d\int_0^t\d s\int_{\R^d}\d y\int_{\R^d}\d z
			\int_{\R^d}\d w\
			\left(\bm{p}_{2s}*f\right)(y-w)|\phi_N(y)||\psi_N(w)|\\
		& = \int_0^t\d s\int_{\R^d}\d \eta \
			\left(\bm{p}_{2s}*f\right)(\eta)\int_{\R^d}\d w\ N^{-d}|\phi(y/N)|\left|\psi\left(\frac{y}{N} -
			\frac{\eta}{N}\right)\right|,
	\end{align*}
	which yields \eqref{eq:AI}.

	In order to prove that $N^{d/2}\mathcal{S}_{N,t}(\psi\,,g)$
	and $N^{d/2}\mathcal{S}_{N,t}(\phi\,,g)$ are asymptotically independent
	as $N\to\infty$ under the condition that the intersection of the supports of $\phi$ and $\psi$
	is Lebesgue-null,  we can replace $\phi$ and $\psi$
	respectively by $a\phi$ and $b\psi$ in \eqref{eq:AI}, where $a,b\in\R$ are arbitrary numbers.
	Thus, it suffices to show that
	\begin{equation} \label{AIvanish}
		\lim_{N\to\infty}\int_0^t\d s\int_{\R^d}\d \eta \
		\left(\bm{p}_{2s}*f\right)(\eta)\,\left(|\phi|*\tilde{|\psi|}\right)(\eta/N) = 0.
	\end{equation}
	By the Cauchy-Schwarz inequality,
	\begin{equation}\label{L_N}
		\sup_{N > 0}\sup_{\eta \in \R^d}
		 \left|\left(|\phi|*\tilde{|\psi|}\right)(\eta/N)\right|\le
		 \|\phi\|_{L^2(\R^d)}\|\psi\|_{L^2(\R^d)}.
	\end{equation}
	It is well known that continuous functions of compact support are dense in $L^2(\R^d)$.
	From this it follows that $\lim_{\|h\|\to 0}\int_{\R^d}|\psi(w + h) - \psi(w)|^2\d w = 0.$
	Therefore, the Cauchy--Schwarz inequality implies that
	\begin{align}\label{eq:lim}
		\lim_{N \to \infty}\left(|\phi|*\tilde{|\psi|}\right)(\eta/N)
		= \int_{\R^d} \left|\phi\left(w \right)\right|\left|\psi\left(w\right)\right|\d w,
	\end{align}
	which vanishes since the intersection of the supports of $\phi$ and $\psi$ is assumed
	to have zero Lebesgue measure. Since $f(\R^d) < \infty$ --- see \eqref{f:finite} ---
	we can deduce \eqref{AIvanish} by combining \eqref{L_N} and \eqref{eq:lim}, using
	the dominated convergence theorem. This completes the proof.
\end{proof}

As was  pointed out earlier, Theorem \ref{th:AI} implies that if $\phi,\psi\in L^2(\R^d)$
have essentially-disjoint compact supports, then for all $a,b\in\R$,
$t\ge0$, and $g\in\lip$,
\[
	\left| \E\exp\left( \e^{ia N^{d/2}\mathcal{S}_{N,t}(\psi,g) +
	ib N^{d/2} \mathcal{S}_{N,t}(\phi,g)}\right) -
	\E\exp\left( \e^{ia N^{d/2}\mathcal{S}_{N,t}(\psi,g)}\right)\E\left(
	\e^{ib N^{d/2} \mathcal{S}_{N,t}(\phi,g)}\right)\right|\to0
\]
as $N\to\infty$. Equivalently, $N^{d/2}\mathcal{S}_{N,t}(\psi\,,g)$ and
$N^{(d/2}\mathcal{S}_{N,t}(\phi\,,g)$ are asymptotically independent as $N\to\infty$.
Now we bootstrap Theorem \ref{th:AI} from a statement about
two functions (namely, $\phi$ and $\psi$) to one about any number of
functions in $L^2(\R^d)$ that have pairwise disjoint compact supports. In any case,
the end result is the following corollary to Theorem \ref{th:AI}. For simplicity,
let $\text{\rm supp}[h]$ denote the support of the function $h:\R^d\to\R$, and define
${\rm Leb}$ to be the Lebesgue measure on $\R^d$.

\begin{corollary}\label{co:AI}
	Choose and fix $t\ge0$ and $g\in\lip$, and let $m\ge2$
	be an integer. Choose $\psi_1,\ldots,\psi_m\in L^2_c(\R^d)$.
	Then,
	for every $a_1,\ldots,a_m\in\R$,
	\begin{align}\label{eq:co:AI}
	\begin{aligned}
		& \left| \E\left[ \e^{i\sum_{j=1}^m a_jN^{d/2}\mathcal{S}_{N,t}(\psi_j\,,g)}\right] -
		\prod_{j=1}^m\E\left[ \e^{ia_jN^{d/2}\mathcal{S}_{N,t}(\psi_j\,,g)}\right]\right| \\
		&\qquad \qquad\lesssim \sum_{k=2}^m\sum_{j = 1}^{k-1}|a_ja_k|\int_0^t\d s\int_{\R^d}\d \eta \
			\left(\bm{p}_{2s}*f\right)(\eta)\left(|\psi_{j}|*|\tilde{\psi}_k|\right)\left(\frac{\eta}{N}\right),
		\end{aligned}
	\end{align}
	uniformly for all $N>0$, and the implied constant is equal to the  implied constant of \eqref{eq:AI}
	and hence does not depend on $(m\,,a_1\,,\ldots,a_m\,,\psi_1\,,\ldots,\psi_m\,,N)$.
    Moreover, suppose that $\psi_1,\ldots,\psi_m\in L^2_c(\R^d)$
	satisfy the following condition:
	\begin{equation}\label{delta:S}
		\text{\rm Leb}\left( \text{\rm supp}[\psi_j]\cap \text{\rm supp}[\psi_k]\right)
		=0\qquad\text{for all $1\le j\neq k\le m$}.
	\end{equation}
    Then, $N^{d/2}\mathcal{S}_{N,t}(\psi_j\,,g), j = 1, \ldots, m$  are asymptotically independent
    as $N\to\infty$.
\end{corollary}


\begin{proof}
	Let $\mathcal{Y}_j := N^{d/2}a_j \mathcal{S}_{N,t}(\psi_j\,,g) = N^{d/2}
	\mathcal{S}_{N,t}(a_j\psi_j\,,g)$ for $j=1,\ldots,m$.
	Define $\mathcal{S}_k := \sum_{j=1}^k\mathcal{Y}_j$,
	$\Psi_k := \sum_{j=1}^ka_j\psi_j$ for every $k=1,\ldots,m$.
	Observe that $\mathcal{S}_k=\mathcal{S}_{k-1}+\mathcal{Y}_k$,
	$\mathcal{S}_{k-1}=N^{d/2}\mathcal{S}_{N,t}(\Psi_{k-1}\,,g)$, and
	$\Psi_k,\psi_{k+1},\ldots,\psi_m\in L^2(\R^d)$
	have compact supports that are pairwise disjoint (for all  {$k=2,\ldots,m$}) { if \eqref{delta:S} holds}.
	In particular, if we set $[m]:=\{1\,,\ldots,m\}$, then we may deduce
	from Theorem \ref{th:AI} the existence of a real number $L>0$
	--- not depending on $(\psi_1\,,\ldots,\psi_k\,,N)$ --- such that
	\begin{align*}
		& \left| \E\left[\e^{i\mathcal{S}_k}\right]\prod_{\ell\in[m]\setminus[k]}
			\E\left[\e^{i\mathcal{Y}_\ell}\right]- \E\left[\e^{i\mathcal{S}_{k-1}}\right]
			\prod_{\ell=k}^m\E\left[\e^{i\mathcal{Y}_\ell}\right]\right|
			\le
			\left| \E\left[\e^{i\mathcal{S}_k}\right] - \E\left[\e^{i\mathcal{S}_{k-1}}\right]
			\E\left[\e^{i\mathcal{Y}_k}\right]\right| \\
		&\hskip2in\le L|a_k| \int_0^t\d s\int_{\R^d}\d \eta \
			\left(\bm{p}_{2s}*f\right)(\eta)\left(|\Psi_{k -1}|*|\tilde{\psi}_k|\right)\left(\frac{\eta}{N}\right),
	\end{align*}
	uniformly for all integers $k=2,\ldots,m$.
	Next, we may write things as a telescoping sum as follows:
	\begin{align*}
		\left| \E\left[\e^{i\mathcal{S}_m}\right] - \prod_{j=1}^m
			\E\left[\e^{i\mathcal{Y}_j}\right]\right|
			&=\left|\sum_{k=2}^m\left\{
			\E\left[\e^{i\mathcal{S}_k}\right]\prod_{\ell\in[m]\setminus[k]}
			\E\left[\e^{i\mathcal{Y}_\ell}\right]- \E\left[\e^{i\mathcal{S}_{k-1}}\right]
			\prod_{\ell=k}^m\E\left[\e^{i\mathcal{Y}_\ell}\right]\right\}\right|\\
		&\le L\sum_{k=2}^m|a_k|\int_0^t\d s\int_{\R^d}\d \eta \
			\left(\bm{p}_{2s}*f\right)(\eta)\left(|\Psi_{k -1}|*|\tilde{\psi}_k|\right)\left(\frac{\eta}{N}\right), \\
        & \leq L\sum_{k=2}^m\sum_{j = 1}^{k-1}|a_ja_k|\int_0^t\d s\int_{\R^d}\d \eta \
			\left(\bm{p}_{2s}*f\right)(\eta)\left(|\psi_{j}|*|\tilde{\psi}_k|\right)\left(\frac{\eta}{N}\right),
	\end{align*}
     which proves \eqref{eq:co:AI}.

    The asymptotical independence property of the random variables $N^{d/2}\mathcal{S}_{N,t}(\psi_j\,,g), j = 1, \ldots, m$
    as $N\to\infty$ under condition \eqref{delta:S} follows \textcolor{red}{from} the same arguments as in the proof of Theorem \ref{th:AI}.
\end{proof}

\begin{remark}
	The last portion of the proof used a method that involves telescoping sums.
	That method was introduced first in 1959
	by Volkonskii and Rozanov \cite{VR} in order to establish
	asymptotic independence for strongly-mixing sequences.
	For a modern, comprehensive, exposition see Bradley \cite[Corollary 1.13, p.\ 32]{Bradley}.
\end{remark}

\section{Proof of Theorems \ref{th:1.1} and \ref{th:1.2}}
\subsection{Convergence in a special case}

Choose and fix some $t\ge0$, $g\in\lip$, $a \in \R$ and $y', y\in\R^d$ such that
$ y'_j \leq y_j $ for all $j=1, \ldots, d$.
 Let
\begin{equation}\label{Q(r)}
	Q(r) =Q(a,r\,;y',y)  := [a\,,a+ r(y_1-y'_1)]\times [y'_2\,,y_2] \times \cdots\times[y'_d\,,y_d]
	\qquad\text{for every $r\ge0$}.
\end{equation}
For every $N>0$, let us define a one-parameter stochastic process
$X_N:=\{X_N(r)\}_{r\ge0}$ as follows:
\[
	X_N(r) := N^{d/2} \mathcal{S}_{N,t}\left(\bm{1}_{Q(r)}\,,g\right)
	\qquad\text{for every $r\ge0$ and $N>0$}.
\]
We define also a one-parameter process $X$ via
\[
	X(r) := \Gamma_t\left(\bm{1}_{Q(r)}\,,g\right)\qquad\text{for every
	$r\ge0$}.
\]

The main goal of this section is to prove the following special case of
Theorems \ref{th:1.1} and \ref{th:1.2}:
\begin{equation}\label{X(1)}
	X_N(1) \xrightarrow{\text{\rm d}\,} X(1)
	\quad\text{as $N\to\infty$.}
\end{equation}
This is a very special case of Lemma \ref{lem:2}, but we will see later on that
Lemma \ref{lem:2} is also a consequence of \eqref{X(1)}. Unfortunately, we do not know
of a direct proof of \eqref{X(1)} that is simple to present. Fortunately,  it is not so hard
to prove the following more general result, as it rests on facts from the
general theory of L\'evy processes.
Here and throughout the symbol $\xrightarrow{\text{\rm fdd}}$
refers to weak convergence of finite-dimensional distributions.

\begin{proposition}\label{pr:X}
	$X_N \xrightarrow{\text{\rm fdd}}X$ as $N\to\infty$.
\end{proposition}

In the first step of the proof of Proposition \ref{pr:X} we identify
the limiting object as a particularly-simple L\'evy process.

\begin{lemma}\label{lem:X=BM}
	$X$ is a one-dimensional Brownian motion with variance
	$\mathbf{B}_t(g\,,g)\prod_{j=1}^d (y_j-y'_j)$.
\end{lemma}

\begin{proof}
	If $r,R\ge 0$, then \eqref{gamma} and \eqref{def:B_t} together imply that
	\[
		\Cov\left[ X(r)\,,X(R)\right] = \mathbf{B}_t(g\,,g) \left(
		\bm{1}_{Q(r)}\,,\bm{1}_{Q(R)}\right)_{L^2(\R^d)}
		=\mathbf{B}_t(g\,,g)\times\min(r\,,R)\times\prod_{j=1}^d (y_j-y'_j).
	\]
	Since $\Gamma$ is a centered Gaussian process, so is $X$.
	This completes the proof.
\end{proof}

Next we have the following uniform tightness result.

\begin{lemma}\label{lem:uniform:tight}
	The laws of $\{X_N(r)\}_{N>0}$ are tight uniformly over all $r\in[0\,,1]$;
	in fact,
	\begin{equation}\label{eq:X:tight}
		\sup_{r\in[0,1]}\sup_{N>0}\E\left( |X_N(r)|^k\right)\le
		\prod_{j=1}^k (y_j-y_j')^{k/2}
		\qquad\text{for every real number $k\ge2$}.
	\end{equation}
\end{lemma}

\begin{proof}
	According to Theorem \ref{th:S}, for every $k\ge2$,
	\[
		\sup_{N>0}\|X_N(r)\|_k^2 \lesssim\| \bm{1}_{Q(r)}\|_{L^2(\R^d)}^2
		=  r\prod_{j=1}^d (y_j-y'_j),
	\]
	where the implied constant does not depend on $r$. This implies \eqref{eq:X:tight}.
	To finish, we apply Chebyshev's inequality and \eqref{eq:X:tight}
	in order to see that
	$\sup_{r\in[0,1]}\sup_{N>0}\P\{ |X_N(r)| > \ell\}=o(1)$ as $\ell\to\infty$.
	This implies the desired uniform tightness.
\end{proof}

{
\begin{proof}[Proof of Proposition \ref{pr:X}]
    Without loss of generality, we restrict the processes $X_N$ and $X$ to $r\in [0, 1]$.
    We first apply Theorem \ref{th:S} and
	stationarity in order to see that there exists a real number $C>0$
	such that, for every $k\ge2$,
	\begin{equation}\label{Xle}
		\sup_{N>0}\|X_N(R) - X_N(r) \|_k \le C\left\|\bm{1}_{Q(R)} - \bm{1}_{Q(r)}\right\|_{L^2(\R^d)}
		=C'  \sqrt{R-r},
	\end{equation}
	uniformly for all $R\ge r\ge0$,
	where $C'=C\prod_{j=1}^d (y_j-y'_j)^{{ 1/2}}$,
	and $C$ does not depend on $r$ and $R$.
	Lemma \ref{lem:uniform:tight} and \eqref{Xle} imply that $\{X_N(r)\}_{r\in [0, 1]}$ is tight in $C[0, 1]$.
	Hence,
	for every unbounded sequence $0<N_1<N_2<\cdots$ there exists a
	subsequence $N'=\{N_n'\}_{n=1}^\infty$ and a process
	$Y=\{Y(r)\}_{r\in[0,1]}$ such that
	\begin{equation}\label{XY:Q}
		X_{N'_n}\xrightarrow{\text{\rm C[0, 1]}\,} Y\qquad\text{as $n\to\infty$}.
	\end{equation}
		As it turns out,  $\{Y(r)\}_{r\in\mathbb{Q}\cap[0,1]}$
	\emph{is} a rather nice stochastic process. In fact, we have the following.\footnote{Caveat.
		Infinitely-divisible processes need not be L\'evy processes, as the latter
		processes must be c\`adl\`ag as well.}
    \bigskip

	\noindent\textbf{Claim A.} \emph{We can realize $Y=\{Y(r)\}_{r\in\cap[0,1]}$
	as {a process with stationary and independent increments} such
	that $Y(0)=0$ and $\E[Y(r)]=0$ for all $r\in[0\,,1]$. } \bigskip

	In order to prove Claim A, let us choose and fix
	an integer $M\ge1$. An application of Theorem \ref{th:S}
	reveals that $X_N(0)=0$; and Corollary \ref{co:AI} ensures that, whenever
	$0=:r_0 < r_1 < \cdots < r_M$,
	\[
		\bigg\{ X_N(r_{i+1})-X_N(r_i)\bigg\}_{i=0}^{M-1}
		\quad \text{are asymptotically independent as $N\to\infty$.}
	\]
    Hence the random variables $Y(r_1)$, $Y(r_2) - Y(r_1)$, $\cdots, Y(r_M) - Y(r_{M-1})$
    are independent for $(r_i)_{i=1}^{M} \subset \mathbb{R}$.
	Moreover, since $u(t)$ is spatially stationary, the law of $X_N(r_{i+1})-X_N(r_i)$
	is the same as the distribution of $X_N(r_{i+1}-r_i)$
	for every $i=0,\ldots,M-1$, which implies that $Y(r_{i+1})-Y(r_i)$ has the same distribution as $Y(r_{i+1} - r_{i})$.
    Therefore, $Y=\{Y(r)\}_{r\in[0,1]}$
	is an infinitely-divisible process with stationary increments.
    It remains to prove that
	$\E[Y(r)]=0$ for all $r$, but this follows from
	the fact that $\E[X_N(r)]=0$ for all $N,r>0$,
	and uniform integrability which is assured by
	\eqref{eq:X:tight}. These remarks together prove Claim A.
	\bigskip

	\noindent\textbf{Claim B.} \emph{The process $Y=\{Y(r)\}_{r\in[0,1]}$ is
	a Brownian motion  normalized
	such that $\Var[Y(1)]=\mathbf{B}_t(g\,,g)\prod_{j=1}^d(y_j-y'_j)$.}
	\bigskip

	L\'evy proved a long time ago that
	the only continuous, mean-zero L\'evy process is Brownian motion. This is
	in fact an immediate consequence
	of the L\'evy--Khintchine formula; see Bertoin \cite{Bertoin}. Therefore, Claim B
	is proved once we show that the variance of $Y(1)$
	is as stated. But that variance formula follows at once from Proposition \ref{pr:Cov:asymp} { and  uniform integrability  assured by
	\eqref{eq:X:tight}}.
	This proves Claim B.

	We are ready to complete the proof of Proposition \ref{pr:X}.

	So far, we have proved that for every unbounded sequence $\{N_n\}_{n=1}^\infty$
	there exists a further subsequence $\{N'_n\}_{n=1}^\infty$ such that the
	distributions of  $X_{N'_n}$ converge to those of a Brownian motion $Y$ in the space $C[0, 1]$ as $n\to\infty$,
	and the speed of that Brownian motion  is always $\mathbf{B}_t(g\,,g)\prod_{j=1}^d(y_j-y'_j)$.
	Lemma \ref{lem:X=BM} tells us that the law of $Y$ is the same as the law of
	$X$
	regardless of the choice of the original subsequence $\{N_n\}_{n=1}^\infty$. This proves  Proposition \ref{pr:X}.
\end{proof}
}
\subsection{Convergence of f.d.d.s}
The following is a first key step in the proofs of both Theorem \ref{th:1.1}
and Theorem \ref{th:1.2}, and is the main result of this subsection.

\begin{proposition}\label{pr:1}
	For every $t\ge0$ and for every $\psi\in L^2(\R^d)$
	and $g\in\lip$,
	\begin{equation}\label{eq:1}
		N^{d/2} \mathcal{S}_{N,t}(\bullet\,,g) \xrightarrow{\text{\rm fdd}}\Gamma_t(\bullet\,,g)
		\qquad\text{and}\qquad
		N^{d/2}\mathcal{S}_{N,t}(\psi\,,\bullet)
		\xrightarrow{\text{\rm fdd}}
		\Gamma_t(\psi\,,\bullet),
	\end{equation}
	as $N\to\infty$.
\end{proposition}

First we verify the following one-dimensional
version of Proposition \ref{pr:1}.

\begin{lemma}\label{lem:2}
	For every $t\ge0$ and for every $\psi\in L^2(\R^d)$
	and $g\in\lip$,
	\begin{equation}\label{eq:2}
		N^{d/2} \mathcal{S}_{N,t}(\psi\,,g) \xrightarrow{\text{\rm d}\,}\Gamma_t(\psi\,,g)
		\qquad\text{as $N\to\infty$}.
	\end{equation}
\end{lemma}

Proposition \ref{pr:1}  follows at once from Lemma \ref{lem:2}
and the following simple conditional result.

\begin{lemma}\label{lem:fdd}
	If Lemma \ref{lem:2} is true, so is Proposition \ref{pr:1}.
\end{lemma}

\begin{proof}
	Choose and fix some $g\in\lip$.  Cram\'er–Wold theorem
	assures us that the
	first assertion of \eqref{eq:1} is equivalent
	to the statement that for every $a_1,\ldots,a_m\in\R$
	and $\psi_1,\ldots,\psi_m\in L^2(\R^d)$,
	\begin{equation}\label{eq:fdd:1}
		N^{d/2}\sum_{i=1}^m a_i \mathcal{S}_{N,t}\left(\psi_i\,,g\right)
		\xrightarrow{\text{\rm d\,}}
		\sum_{i=1}^m a_i\Gamma_t\left( \psi_i\,,g\right).
	\end{equation}
	Define $\psi:=\sum_{i=1}^m a_i\psi_i\in L^2(\R^d)$
	and use bilinearity to see that the left-hand side of
	\eqref{eq:fdd:1} is equal to $N^{d/2}\mathcal{S}_{N,t}(\psi\,,g)$
	whereas the right-hand side of
	\eqref{eq:fdd:1} is equal to $\Gamma_t(\psi\,,g)$.
	Therefore, eq.\ \eqref{eq:fdd:1} --- and hence also
	the first assertion of \eqref{eq:1} --- both follow from \eqref{eq:2}.
	The second claim in \eqref{eq:1} is proved similarly.
\end{proof}

Thus, it remains to demonstrate Lemma \ref{lem:2}. That proof requires
some effort which we distribute in parts. The first portion
of that proof is a  ``density lemma'', that is presented next.

\begin{lemma}\label{lem:dense}
	Suppose that $\mathscr{E}$ is a dense subset of $L^2(\R^d)$
	such that \eqref{eq:2} holds for every $\psi\in\mathscr{E}$.
	Then, \eqref{eq:2} is valid  for all $\psi\in L^2(\R^d)$.
\end{lemma}

\begin{proof}
	Choose and fix $(\psi\,,g)\in L^2(\R^d)\times\lip$.
	For every $\varepsilon>0$ we can find $\phi\in\mathscr{E}$
	such that $\|\phi-\psi\|_{L^2(\R^d)}\le\varepsilon$. According to Theorem \ref{th:S},
	there exists a real number $K$ --- independent of $\psi$ and $\phi$ --- such that
	\[
		 \sup_{N>0}  \left\| N^{d/2}\mathcal{S}_{N,t}(\psi\,,g) - N^{d/2}\mathcal{S}_{N,t}(\phi\,,g)\right\|_2 =
		 \sup_{N>0} \left\| N^{d/2}\mathcal{S}_{N,t}(\psi-\phi\,,g) \right\|_2
		 \le K\varepsilon.
	\]
	Let $H:\R\to\R$ be bounded and Lipschitz continuous. By virtue of the
	definition of $\mathscr{E}$,
	\[
		\lim_{N\to\infty}\Delta_N=0,
		\quad\text{where}\quad
		\Delta_N :=\left| \E \left[ H( N^{d/2} \mathcal{S}_{N,t}(\phi\,,g)) \right]
		- \E \left[ H (\Gamma_t(\phi\,,g)) \right]\right|.
	\]
	Now,
	\[
		\left| \E \left[H( N^{d/2} \mathcal{S}_{N,t}(\phi\,,g)) \right] -
		\E \left[ H( N^{d/2} \mathcal{S}_{N,t}(\psi\,,g)) \right] \right|
		\le K\lip(H)\varepsilon,
	\]
	and
	\begin{align*}
		\left| \E \left[ H( \Gamma_t(\phi\,,g)) \right] -
			\E \left[ H( \Gamma_t(\psi\,,g)) \right] \right|
			&\le \lip(H)\left\| \Gamma_t(\phi\,,g) - \Gamma_t(\psi\,,g)\right\|_2
			=\lip(H)\left\| \Gamma_t(\phi-\psi\,,g)\right\|_2\\
		&=\lip(H)\|\phi-\psi\|_{L^2(\R^d)}\sqrt{\mathbf{B}_t(g\,,g)}
			\le\lip(H)\sqrt{\mathbf{B}_t(g\,,g)}\,\varepsilon \\
		&=: L\lip(H)\varepsilon,
	\end{align*}
	for a real number $L>0$ that is independent of $\psi$ and $\phi$.
	Thus, it follows from the triangle inequality that
	\[
		\left| \E \left[ H( N^{d/2}\mathcal{S}_{N,t}(\psi\,,g)) \right] -
		\E \left[ H(\Gamma_t(\psi\,,g)) \right]
		\right| \le\Delta_N+(K+L)\lip(H)\varepsilon.
	\]
	Let $N\to\infty$ and $\varepsilon\to0$, in this order, to see that
	the quantity in the left-hand side of the above tends to zero as $N\to\infty$.
	Because bounded, Lipschitz-continuous functions are convergence-determining,
	this suffices to establish the asserted weak convergence
	of $N^{d/2}\mathcal{S}_{N,t}(\psi\,,g)$ to $\Gamma_t(\psi\,,g)$.
\end{proof}

In light of Lemma \ref{lem:fdd}, it suffices to prove Lemma \ref{lem:2} for a dense class
 $\mathscr{E}$ in $L^2(\R^d)$;
Proposition \ref{pr:1} follows \emph{a fortiori}.

\begin{proof}[Proof of Lemma \ref{lem:2}]
	Define $\mathscr{E}$
	to be the collection of all functions $\psi\in L^2_c(\R^d)$
	that have the form,
	\begin{equation}\label{psi:rep}
		\psi=\psi_1+\cdots+\psi_m,
		\quad\text{where}\quad
		\psi_i(x) = a_i\bm{1}_{[y^i,\,z^i]}(x)\quad\text{for all $x\in\R^d$},
	\end{equation}
	$m\ge1$ is an integer, $a_1,\ldots,a_m\in\R\setminus\{0\}$,  $y^1,\ldots,y^m\in\R^d$,
	$z^1,\ldots,z^m\in\R^d$, with $y^j \le z^j$;
	and
	\[
		\text{\rm Leb}\left( [y^i\,,z^i] \cap [y^j\,,z^j]\right)=0
		\quad\text{whenever $1\le i\neq j\le m$}.
	\]
	It is easy to see that $\mathscr{E}$ is dense in $L^2(\R^d)$; this is
	an exercise in the theory of Lebesgue integration. Therefore, Lemma
	\ref{lem:dense} will imply Lemma \ref{lem:2} once we prove that
	$N^{d/2} \mathcal{S}_{N,t}(\psi\,,g) \xrightarrow{\text{\rm d}\,}
	\Gamma_t(\psi\,,g)$, as $N\to\infty$, for every $t\ge0$,
	$\psi\in\mathscr{E}$, and $g\in\lip$.
	With this aim in mind, let us choose and fix some $t\ge0$, $\psi\in\mathscr{E}$,
	and $g\in\lip$, and assume that $\psi$ has the representation in \eqref{psi:rep}.
	By bilinearity,
	\[
		N^{d/2} \mathcal{S}_{N,t}(\psi\,,g) = N^{d/2}\sum_{i=1}^m
		\mathcal{S}_{N,t}( \psi_i\,,g) =: \sum_{i=1}^m X_{i,N}\qquad\text{a.s.,}
	\]
	where $X_{i,N} := N^{d/2}\mathcal{S}_{N,t}(\psi_i\,,g)$.
	Corollary \ref{co:AI} ensures that $\{X_{i,N}\}_{i=1}^m$
	describes an asymptotically independent sequence as $N\to\infty$;
	and Proposition \ref{pr:X}  implies that
	\[
		X_{i,N}\xrightarrow{\text{d}\,}\Gamma_t(\psi_i\,,g)
		\qquad\text{as $N\to\infty$,
		for every $i=1,\ldots,m$}.
	\]
	The asserted asymptotic independence then implies that
	\[
		N^{d/2} \mathcal{S}_{N,t}(\psi\,,g) \xrightarrow{\text{d}\,} Y_1+\cdots+Y_m
		\qquad\text{as $N\to\infty$},
	\]
	where $Y_1,\ldots,Y_m$ are independent, and the distribution of $Y_i$
	is the same as that of $\Gamma_t(\psi_i\,,g)$ for every $i=1,\ldots,m$.
	Because the supports of
	the $\psi_i$'s are disjoint, $\Gamma_t(\psi_1\,,g),\ldots,\Gamma_t(\psi_m\,,g)$
	are uncorrelated, hence independent, Gaussian random variables.
	In particular, we can rewrite the preceding in the following equivalent form:
	\[
		N^{d/2} \mathcal{S}_{N,t}(\psi\,,g) \xrightarrow{\text{d}\,}
		\Gamma_t(\psi_1\,,g)+\cdots+\Gamma_t(\psi_m\,,g)\qquad
		\text{as $N\to\infty$},
	\]
	This fact and the linearity of $\phi\mapsto\Gamma_t(\phi\,,g)$ together imply that
	$N^{d/2}\mathcal{S}_{N,t}(\psi\,,g)$ converges in distribution to $\Gamma_t(\psi\,,g)$
	for every $\psi\in\mathscr{E}$, as was desired. This concludes the proof of
	Lemma \ref{lem:2}.
\end{proof}

\subsection{Metric entropy}\label{subsec:ME}
Let $(\mathscr{T},\mathsf{d})$ be a compact metric space and
$X:=\{X(t)\}_{t\in \mathscr{T}}$ a stochastic process indexed by $\mathscr{T}$.
Define $\Delta(\mathscr{T}) := \max_{s,t\in\mathscr{T}}\mathsf{d}(s\,,t)$
to be the diameter of $\mathscr{T}$, and set
\[
	\Psi(u) := \max_{s,t\in\mathscr{T}}\P\left\{|X_s-X_t|> \mathsf{d}(s\,,t)u\right\}
	\qquad\text{for all $u\ge0$}.
\]
We may now define a ``tail probability function,''
\[
	\bm{\tau}(\lambda) := \int_0^\infty\left(\lambda\Psi(u)\wedge 1\right)\d u
	\qquad\text{for all $\lambda>0$}.
\]
It is known generally that, if $\bm{\tau}(\lambda)\to0$ sufficiently rapidly as $\lambda\to0$
then $X$ has a continuous modification. The following result is a concrete version of a
family of known results in the literature, particularly well worked out for Gaussian processes $X$
(see Chapter 6 of Marcus and Rosen \cite{MR}, for example). Here and throughout, define
$\bm{N}_{\mathscr{T}}$ to be the \emph{metric entropy} of $(\mathscr{T}\!,d)$. That is,
for every $r>0$,
$\bm{N}_{\mathscr{T}}(r):=$ the minimum number of open $\mathsf{d}$-balls of radius $r$
needed to cover $\mathscr{T}$.

\begin{theorem}\label{th:ME}
	For every finite set $\mathscr{S}\subset\mathscr{T}$
	and for all $\delta\in(0\,,\Delta(\mathscr{S}) )$,
	\begin{equation}\label{eq:ME}
		\E\left(\max_{\substack{s,t\in\mathscr{S}:\\
		\mathsf{d}(s,t)\le\delta}} |X_s-X_t|\right) \le 32\int_0^{\delta/4}
		\bm{\tau}\left(\left| \bm{N}_{\mathscr{S}}(r) \right|^2\right)\d r.
	\end{equation}
	In particular, if $\int_{0^+}
	\bm{\tau}(|\bm{N}_{\mathscr{T}}(r)|^2)\,\d r<\infty$, then
	$X$ has a continuous modification.
\end{theorem}

The proof involves a more-or-less standard ``chaining
argument.''
We include it in order to demonstrate the ubiquitous nature
of the multiplicative constant ``32'' in front of the metric entropy integral
on the right-hand side of \eqref{eq:ME}.

First, we establish two elementary lemmas.

\begin{lemma}\label{lem:entropy:1}
	Let $\Theta\subset \mathscr{T}\times \mathscr{T}$ be a finite set of cardinality $|\Theta|$. Then,
	\[
		\E\left[\max_{(s,t)\in\Theta}|X_t-X_s|\right]
		\le \bm{\tau}(|\Theta|)\cdot \sup_{(s, t)\in \Theta}\mathsf{d} (s\,, t).
	\]
\end{lemma}

\begin{proof}
	For every $u>0$,
	\[
		\P\left\{ \max_{(s,t)\in \Theta}
		\left|\frac{X_t-X_s}{\mathsf{d}(s\,,t)}\right| > u \right\}
		\le 1\wedge\sum_{(s,t)\in \Theta}\P\left\{
		\left|\frac{X_t-X_s}{\mathsf{d}(s\,,t)}\right| > u \right\}
		\le |\Theta| \Psi(u)\wedge 1.
	\]
	where $0\div 0 :=0$.
	Integrate $[\d u]$ to see that
	\[
		\E\left(\max_{(s,t)\in \Theta}
		\left|\frac{X_t-X_s}{\mathsf{d}(s\,,t)}\right|\right)\le \bm{\tau}(|\Theta|).
	\]
	This implies the lemma.
\end{proof}

Next we apply Lemma \ref{lem:entropy:1} to improve itself.

\begin{lemma}\label{lem:entropy}
	If $\mathscr{T}$ is a finite set, then
	\[
		\max_{t_0\in \mathscr{T}}\E\left( \max_{t\in \mathscr{T}}|X_t-X_{t_0}|\right)
		\le 8 \int_0^{\Delta(\mathscr{T})/4}\bm{\tau}(\bm{N}_{\mathscr{T}}(r))\,\d r.
	\]
\end{lemma}

\begin{proof}
	Let $\bm{P}_{\mathscr{T}}$ denote the \emph{Kolmogorov capacity}
	of $(\mathscr{T}\!,\mathsf{d})$. That is, for every $r>0$,
	$\bm{P}_{\mathscr{T}}(r):=$ the greatest integer $m\ge1$
	such that there exist points $t_1,\ldots,t_m\in\mathscr{T}$
	such that $\mathsf{d}(t_i\,,t_j)>r$ whenever $i\neq j$. It is well known that
	\begin{equation}\label{eq:N:C}
		\bm{N}_{\mathscr{T}}(2r) \le \bm{P}_{\mathscr{T}}(r) \le
		\bm{N}_{\mathscr{T}}(r/2)\qquad\text{for all $r>0$};
	\end{equation}
	see for example Marcus and Rosen \cite[Lemma 6.1.1]{MR}.

	For every integer $n\ge0$ define
	\[
		\varepsilon_n := 2^{-n}\Delta(\mathscr{T})
		\quad\text{and}\quad
		K_n := \bm{P}_{\mathscr{T}}(\varepsilon_n).
	\]
	One can see readily that $1=K_0\le K_1\le K_2\le\ldots$\,.

	The definition of Kolmogorov capacity ensures that for every integer
	$n\ge 0$ we can find a finite set $\mathscr{T}_n\subset \mathscr{T}$ such that:
	\begin{compactitem}
		\item $|\mathscr{T}_n|=K_n$, where $|\,\cdot|$ denotes cardinality;
		\item $\mathsf{d}(u\,,v)>\varepsilon_n$ for all distinct pairs of points
			$u,v\in \mathscr{T}_n$;
		\item $\inf_{s\in \mathscr{T}_n}\mathsf{d}(s\,,t) \le \varepsilon_n$ for all
			$t\in \mathscr{T}$; and
		\item There exists an integer
			$M=M(\mathscr{T}\!,\mathsf{d})\ge 1$ such that
			$\mathscr{T}_n=\mathscr{T}$ for all $n\ge M$.
	\end{compactitem}

	For every $n\ge0$ let $\pi_n$ denote the projection of $\mathscr{T}$ onto $\mathscr{T}_n$;
	more precisely, $\pi_n(t)$ denotes the point in $\mathscr{T}_n$ that is closest to $t$
	for every $t\in \mathscr{T}$.
	If there are many such points then we break the ties in some arbitrary fashion.
	Since $\mathscr{T}_0$
	is a singleton we can write it as $\mathscr{T}_0=\{t_0\}$ and observe that
	$\pi_0(t)=t_0$ for all $t\in \mathscr{T}$.
	Also, observe that $t_0\in \mathscr{T}$ can be chosen in a completely arbitrary manner,
	without altering any of the preceding statements.

	Since $\mathscr{T}_n=\mathscr{T}$ for all $n\ge M$ it follows that $\pi_n(t)=t$ for every $n\ge M$.
	Thus, to every $t\in \mathscr{T}$ we can associate a ``chain'' $\{t_i\}_{i=0}^\infty$ of points as follows:
	Set $t_n=\pi_M(t) = t$ for all $n\ge M$, and then
	recursively define $t_{i-1} =\pi_{i-1}(t_i)$ for all $i=M,\ldots,1$. This
	sequence ends with $t_0$ --- the unique element of $\mathscr{T}_0$ --- and therefore,
	$X_t - X_{t_0} = \sum_{i=0}^\infty( X_{t_{i+1}} - X_{t_i}).$
	Clearly, all of the summands vanish after the $M$-th term.
	In any case, it follows that
	\[
		\left| X_t - X_{t_0} \right| \le \sum_{i=0}^\infty
		\max_{u\in \mathscr{T}_{i+1}}\left| X_u - X_{\pi_i(u)} \right|,
	\]
	uniformly for all $t\in \mathscr{T}$. Since the right-hand side
	does not depend on the point $t$, Lemma \ref{lem:entropy:1} implies that
	\[
		\E\left( \max_{t\in \mathscr{T}} |X_t - X_{t_0}|\right) \le \sum_{i=0}^\infty
		\bm{\tau}(|\mathscr{T}_{i+1}|) \varepsilon_i
		= \sum_{i=0}^\infty \bm{\tau}
		\left( \bm{P}_{\mathscr{T}}(\varepsilon_{i+1})\right)
		\varepsilon_i;
	\]
	we have used the {fact} that
	$|\mathscr{T}_j| = \bm{P}_{\mathscr{T}}(\varepsilon_j)$.
	Since $\varepsilon_i = 4(\varepsilon_{i+1}-\varepsilon_{i+2})$
	for every $i\ge0$, we can then write
	\begin{align*}
		\E\left(\max_{t\in \mathscr{T}}|X_t - X_{t_0}|\right) &\le 4\sum_{i=0}^\infty
			\int_{\varepsilon_{i+2}}^{\varepsilon_{i+1}}
			\bm{\tau}\left( \bm{P}_{\mathscr{T}}(\varepsilon_{i+1})\right)\d r
			\le 4\sum_{i=0}^\infty\int_{\varepsilon_{i+2}}^{\varepsilon_{i+1}}
			\bm{\tau}\left( \bm{P}_{\mathscr{T}}(r)\right)\d r\\
		&= 4\int_0^{\Delta(\mathscr{T})/2}
			\bm{\tau}\left( \bm{P}_{\mathscr{T}}(r)\right)\d r
			\le 4\int_0^{\Delta(\mathscr{T})/2}
			\bm{\tau}\left(\bm{N}_{\mathscr{T}}(r/2)\right)\d r;
	\end{align*}
	see  \eqref{eq:N:C}. Because $t_0\in \mathscr{T}$ is arbitrary,
	this and a change of variables together yield the lemma.
\end{proof}

We are ready to prove Theorem \ref{th:ME}

\begin{proof}[Proof of Theorem \ref{th:ME}]
	We need only verify \eqref{eq:ME}; the continuity portion follows
	from the quantitative bound \eqref{eq:ME} and standard arguments.
	From now on, we may (and will) assume without loss of generality
	that $\mathscr{T}$  is finite and that $\mathscr{S}=\mathscr{T}$.
	Otherwise, restrict the index set of $X$ to $\mathscr{S}$ and
	relabel $\mathscr{S}$ as $\mathscr{T}$ everywhere that follows.

	The remainder of the proof of \eqref{eq:ME} hinges on ``tensorization.''

	Define $\tilde{\mathscr{T}} := \mathscr{T}\times \mathscr{T}$, and endow it with ``product distance,''
	\[
		\tilde{\mathsf{d}}\left( (s\,,t)\,,(s',t')\right) := \mathsf{d}(s\,,s') \vee \mathsf{d}(t\,,t')
		\qquad\text{for every $s,t,t',t'\in \mathscr{T}$}.
	\]
	The product nature of $\tilde{\mathscr{T}}$
	implies that if the balls $B_1,\ldots,B_m$ form an $\varepsilon$-cover for
	$(\mathscr{T}\!,\mathsf{d})$, then certainly the balls
	$\{B_i\times B_j\}_{i,j=1}^m$ form an $\varepsilon$-cover
	for $(\tilde{\mathscr{T}}\!,\tilde{\mathsf{d}})$. Consequently,
	\begin{equation}\label{NN2}
		\bm{N}_{\tilde{\mathscr{T}}}(\varepsilon) \le
		\left[ \bm{N}_{\mathscr{T}}(\varepsilon)\right]^2
		\qquad\text{for every $\varepsilon>0$}.
	\end{equation}

	Consider the stochastic process $\tilde{X}$, indexed by $\tilde{\mathscr{T}}$, as follows:
	\[
		\tilde{X}_{(s,t)} := X_t - X_s\qquad\text{for every $(s\,,t)\in\tilde{\mathscr{T}}$}.
	\]
	We may combine \eqref{NN2} and Lemma \ref{lem:entropy} (applied
	to $\tilde{X}$ in place of $X$) in order
	to see that
	\begin{equation}\label{ENT}
		\max_{\tilde{t}_0\in\tilde{\mathscr{T}}}
		\E\left( \max_{(s,t)\in\tilde{\mathscr{T}}}\left| \tilde{X}_{(s,t)} -
		\tilde{X}_{\tilde{t}_0} \right| \right) \le 8
		\int_0^{\Delta(\tilde{\mathscr{T}})/4}
		\tilde{\bm{\tau}}\left(\left| \bm{N}_{\mathscr{T}}(r)\right|^2\right)\d r,
	\end{equation}
	where  $\tilde{\bm{\tau}}(\lambda) := \int_0^\infty (\lambda\tilde\Psi(u)\wedge 1)\,\d u$
	for every $\lambda>0$, and
	\[
		\tilde\Psi(u):=
		\sup_{(s,t),(s',t')\in\tilde{\mathscr{T}}}\P\left\{
		\left| \tilde{X}_{(s,t)} - \tilde{X}_{(s',t')}\right| >
		\tilde{\mathsf{d}}\left( (s\,,t)\,,(s',t')\right) u\right\}\qquad
		[u>0].
	\]
	Note that
	\begin{align*}
		\tilde\Psi(u) &\le
			\sup_{(s,t),(s',t')\in\tilde{\mathscr{T}}}\P\left\{
			|X_s-X_{s'}| + |X_t-X_{t'}| >
			\tilde{\mathsf{d}}\left( (s\,,t)\,,(s',t')\right) u\right\}\\
		&\le\sup_{(s,t),(s',t')\in\tilde{T}}\P\left\{
			|X_s-X_{s'}|  > \tfrac12
			\tilde{\mathsf{d}}\left( (s\,,t)\,,(s',t')\right) u\right\} \\
		&\hskip1.5in +
			\sup_{(s,t),(s',t')\in\tilde{T}}\P\left\{
			 |X_t-X_{t'}| > \tfrac12
			\tilde{\mathsf{d}}\left( (s\,,t)\,,(s',t')\right) u\right\}.
	\end{align*}
	Therefore,
	$\tilde\Psi(u) \le 2\Psi(u/2)$ for every $u\ge0$, by virtue of the definition of $\tilde{\mathsf{d}}$ {and $\Psi$}.
	In particular,
	\begin{align*}
		\tilde{\bm{\tau}}(\lambda) &\le \int_0^\infty
		\left(2\lambda\Psi(u/2)\wedge 1\right)\d u
		\le 4\bm{\tau}(\lambda) \qquad\text{for all $\lambda>0$}.
	\end{align*}
	Since $\Delta(\tilde{\mathscr{T}}) = \Delta(\mathscr{T})$,
	\eqref{ENT} implies that, for every $\tilde{t}_0\in\tilde{\mathscr{T}}$,
	\[
		\E\left( \max_{(s,t)\in\tilde{\mathscr{T}}}
		\left| \tilde{X}_{(s,t)} - \tilde{X}_{\tilde{t}_0}\right| \right) \le
		32\int_0^{\Delta(\mathscr{T})/4}
		\bm{\tau}\left(\left| \bm{N}_{\mathscr{T}}
		(r)\right|^2\right) \,\d r,
	\]
	We now choose $\tilde{t}_0:=(t_0\,,t_0)$
	for an arbitrary but fixed point $t_0\in \mathscr{T}$.
	For this choice, $\tilde{X}_{\tilde{t}_0}=0$, and the theorem follows.
\end{proof}

\subsection{Proof of Theorems \ref{th:1.1} and \ref{th:1.2}}

We apply Lemma \ref{lem:tightness}  in order to see that
there exist  constants $K,L>0$ such that
\[
	\Psi(u) := \adjustlimits\sup_{N>0}\sup_{\psi\in L^2(\R^d)}\sup_{g\in\lip}
	\P\left\{ N^{d/2} \left| \mathcal{S}_{N,t}(\psi\,,g)\right|
	> \|\psi\|_{L^2(\R^d)} \|g\|_\lip u\right\} \lesssim
	\exp\left\{ -\frac{L \log_+(u)}{\Upsilon
	\left(K \log_+(u)\right)}\right\},
\]
uniformly for $u>0$.

Because the behavior of $\Upsilon$ can
depend on the fine details of the statistics of \eqref{SHE},
it is better to use the simple but general fact that $\lim_{\lambda\to\infty}\Upsilon(\lambda)=0$
in order to see that
\[
	\Psi(u) \lesssim u^{-1/\varepsilon}\qquad\text{for all $\varepsilon\in(0\,,1)$
	and $u>0$}.
\]
Therefore,
\[
	\bm{\tau}(\lambda) := \int_0^\infty\left( \lambda\Psi(u)\wedge1\right)\d u
	\lesssim\int_0^\infty\left(
	\frac{\lambda}{u^{1/\varepsilon}}\wedge 1\right)\d u
	\propto\lambda^\varepsilon
	\qquad\text{for all $\lambda>0$}.
\]
{
Since $\mathscr{F}$ is separable, let $\{r_1, r_2, \ldots\}$ be a dense subset of $\mathscr{F}$ and denote $\mathscr{F}_n= \mathscr{F}\cap \{r_1, r_2, \ldots, r_n\}$.
Now apply Theorem \ref{th:ME} in order to see
that for all $\varepsilon\in(0\,,1)$ there exists $C(\varepsilon)>0$ such that
for all $\delta\ll1$,
\begin{align*}
	\sup_{N>0} N^{d/2}
	\E\left(\max_{\substack{\psi\,,{\Phi}\in\mathscr{F}_n:\\
	\|{\Phi}-\psi\|_{L^2(\R^d)}\le\delta}}
	\left| \mathcal{S}_{N,t}({\Phi}\,,g) - \mathcal{S}_{N,t}(\psi\,,g)\right|\right)
	&\le C(\varepsilon) \int_0^{\delta/4}
	\left[ \bm{N}_{\mathscr{F}_n\!,L^2(\R^d)}(r)\right]^\varepsilon\,\d r\\
	&\le C(\varepsilon) \int_0^{\delta/4}
	\left[ \bm{N}_{\mathscr{F}\!,L^2(\R^d)}(r)\right]^\varepsilon\,\d r.
\end{align*}
By monotone convergence theorem, we let $n\to \infty$ to obtain that
\begin{align*}
	\sup_{N>0} N^{d/2}
	\E\left(\max_{\substack{\psi\,,{\Phi}\in\mathscr{F}:\\
	\|{\Phi}-\psi\|_{L^2(\R^d)}\le\delta}}
	\left| \mathcal{S}_{N,t}({\Phi}\,,g) - \mathcal{S}_{N,t}(\psi\,,g)\right|\right)
		&\le C(\varepsilon) \int_0^{\delta/4}
	\left[ \bm{N}_{\mathscr{F}\!,L^2(\R^d)}(r)\right]^\varepsilon\,\d r.
\end{align*}
}
Similarly, we can see that for all $\varepsilon\in(0\,,1)$ there exists
$C'(\varepsilon)>0$ such that
for all $\delta\ll1$,
\[
	\sup_{N>0} N^{d/2}
	\E\left(\max_{\substack{g\,,G\in\mathscr{G}:\\
	\|G-g\|_\lip\le\delta}}
	\left| \mathcal{S}_{N,t}(\psi\,,g) - \mathcal{S}_{N,t}(\psi\,,G)\right|\right)
	\le C'(\varepsilon) \int_0^{\delta/4}
	\left[ \bm{N}_{\mathscr{G}\!,\lip}(r)\right]^\varepsilon\,\d r.
\]
The above two bounds imply the requisite tightness results. Proposition
\ref{pr:1} and tightness together imply both
Theorems \ref{th:1.1} and \ref{th:1.2}; see \cite[p.58 and Theorem 5.1]{Bil99}.\qed

\subsection{Some examples}
Let us conclude with a few elementary examples of the sorts of classes of functions
that Theorems  \ref{th:1.1} and \ref{th:1.2} refer to.

\begin{example}\label{ex:BS:1}
	For our first example, let us choose and  fix some vector $m\in\R^d_+$ and define
	\[
		\mathscr{F} := \left\{ \bm{1}_{[0,y]}:\, y\in[0\,,m]^d\right\}.
	\]
	Because $\| \bm{1}_{[0,y]}-\bm{1}_{[0,z]}\|_{L^2(\R^d)} =
	| [0\,,y] \triangle [0\,,z]| ^{1/2}\lesssim \|y-z\| ^{1/2}$,
	uniformly for all $y,z\in[0\,,m]^d$, it follows that
	$\bm{N}_{\mathscr{F}\!,L^2(\R^d)}(r)\lesssim r^{-2d}$
	uniformly for all $r\in(0\,,1)$,
	and so $\int_0^1[\bm{N}_{\mathscr{F}\!,L^2(\R^d)}(r)]^\varepsilon\,\d r<\infty$
	for every $\varepsilon\in(0\,,1/(2d))$. Thus, we see that Theorem \ref{th:1.1} implies
	the weak convergence \eqref{W_{N,t}} to the Brownian sheet.
\end{example}

\begin{example}\label{ex:BS:2}
	Suppose $\mathscr{C}$ and $\mathscr{D}$ are compact subsets of $L^2(\R^d)$
	such that $\int_0^1[\bm{N}_{\mathscr{C}\!,L^2(\R^d)}(r)]^\varepsilon\,\d r+
	\int_0^1[\bm{N}_{\mathscr{D}\!,L^2(\R^d)}(r)]^\varepsilon\,\d r<\infty$
	for some $\varepsilon>0$. Define
	\[
		\mathscr{F} := \left\{ C*D:\, C\in \mathscr{C},\, D\in\mathscr{D}\right\},
	\]
	where ``$*$" refers to the convolution of two functions.
	Then by Young's inequality for convolutions,
	\begin{align*}
		\|(C*D)-(c*d)\|_{L^2(\R^d)} &\le \|(C*D)-(C*d)\|_{L^2(\R^d)} +
			\|(C*d)-(c*d)\|_{L^2(\R^d)}\\
		&\le \|C\|_{L^2(\R^d)}\|D-d\|_{L^2(\R^d)} + \|d\|_{L^2(\R^d)}\|C-c\|_{L^2(\R^d)}\\
		&\lesssim \|C-c\|_{L^2(\R^d)} + \|D-d\|_{L^2(\R^d)},
	\end{align*}
	uniformly for every $c,C\in\mathscr{C}$ and $d,D\in\mathscr{D}$. Thus,
	$\bm{N}_{\mathscr{F}\!,L^2(\R^d)}(r)\lesssim
	\bm{N}_{\mathscr{C}\!,L^2(\R^d)}(r)\bm{N}_{\mathscr{D}\!,L^2(\R^d)}(r),$
	uniformly for all $r\in(0\,,1)$. In particular,
	$\int_0^1[\bm{N}_{\mathscr{F}\!,L^2(\R^d)}(r)]^{\varepsilon/2}\,\d r<\infty,$
	thanks to the Cauchy-Schwarz inequality.
\end{example}

\begin{example}\label{ex:Lip:1}
	Choose and fix a $C^1$-function $g$ such that $g$ and $g'$ are Lipschitz. Define
	$g_a(u) := g( u-a )$ for all $a,u\in\R$, and set
	\[
		\mathscr{G} := \bigcup_{a\in[-n,n]}\{g_a\},
	\]
	where $n>0$ is a fixed real number. Since
	$\|g_a-g_b\|_{\lip} \le(\lip(g) + \lip(g'))|b-a|$ for all $a,b\in\R$, it follows that
	$\bm{N}_{\mathscr{G}\!,\lip}(r)\lesssim r^{-1}$, uniformly for all $r\in(0\,,1)$,
	and hence
	$\int_0^1[\bm{N}_{\mathscr{G},\lip}(r)]^\varepsilon\,\d r<\infty$
	for every $\varepsilon\in(0\,,1)$.
\end{example}

\begin{example}\label{ex:Lip:2}
	Choose and fix a $C^1$-function $g$ such that $g$ and $g'$ are Lipschitz.
	This time define $g_{a,b}(u) := bg( u/a )$ for all $b,u\in\R$ and $a>0$, and set
	\[
		\mathscr{G} := \bigcup_{\substack{b\in[-n,n]\\a\in[1/m,m]}}\{g_{a,b}\},
	\]
	where $m>1$ and $n>0$ are fixed real numbers. Because
	$\|g_{a,b}-g_{A,B}\|_{\lip} \lesssim|A-a|+|B-b|,$
	uniformly for all $a,A\in[1/m\,,m]$ and $b,B\in[-n\,,n]$,
	it follows that $\bm{N}_{\mathscr{G}\!,\lip}(r)\lesssim r^{-2}$ uniformly for
	every $r\in(0\,,1)$ and hence
	$\int_0^1[\bm{N}_{\mathscr{G}\!,\lip}(r)]^\varepsilon\,\d r<\infty$
	for every $\varepsilon\in(0\,,1/2)$.
\end{example}

\noindent
{\bf Acknowledgement.} We would like to thank the associate editor and two referees
for their valuable and useful comments.

\end{document}